\documentclass{article}
\usepackage{graphicx} 

\PassOptionsToPackage{backend=biber, sortcites, style=numeric-comp}{biblatex}

\usepackage{styling/basics}
\usepackage{styling/algorithms}
\usepackage{styling/referencing}
\usepackage{styling/math}
\usepackage{styling/theorems}
\usepackage{styling/tikzen}

\RequirePackage{macros/calculus}
\RequirePackage{macros/operators}
\RequirePackage{macros/general_math}
\RequirePackage{macros/spaces}
\RequirePackage{macros/topology}

\pgfplotsset{compat=1.18}

\usepackage{biblatex}
\title{The Influence of an Adjoint Mismatch on the Primal-Dual Douglas-Rachford Method}
\author{Emanuele Naldi\thanks{Dipartimento di Matematica, University of Genova, Via Dodecaneso 35, 16146 Genova, Italy, \email{emanuele.naldi@edu.unige.it}} \and Felix Schneppe\thanks{Center for Industrial Mathematics, Fachbereich 3, University of Bremen, Postfach 33 04 40, 28334 Bremen, Germany, \email{schneppe@uni-bremen.de}}}

\begin{document}

\maketitle

\begin{abstract}
    The primal-dual Douglas-Rachford method is a well-known algorithm to solve optimization problems written as convex-concave saddle-point problems. Each iteration involves solving a linear system involving a linear operator and its adjoint. However, in practical applications it is often computationally favorable to replace the adjoint operator by a computationally more efficient approximation. This leads to an adjoint mismatch. In this paper, we analyze the convergence of the primal-dual Douglas-Rachford method under the presence of an adjoint mismatch. We provide mild conditions that guarantee the existence of a fixed point and find an upper bound on the error of the primal solution. Furthermore, we establish step sizes in the strongly convex setting that guarantee linear convergence under mild conditions. Additionally, we provide an alternative method that can also be derived from the Douglas-Rachford method and is also guaranteed to converge in this setting. Finally, we illustrate our results both for an academic and a real-world inspired example.
\end{abstract}

\section{Motivation}
In many practical applications we try to find a solution to the minimization problem
\begin{equation}\label{eq:minFA+G}
    \min_{x \in X}\, F(Ax)+G(x),
\end{equation}
on a Hilbert space $X$, where $F:Y\to \overline{\RR}$, $G:Y\to \overline{\RR}$ are proper, convex, lower-semicontinuous functions, $Y$ is another Hilbert space and $A: X \to Y$ is a linear and bounded operator. Since problems of this kind are often hard to solve, it can be beneficial to examine the equivalent saddle point problem
\begin{equation}\label{eq:saddlepoint}
    \min_{x \in X}\, \max_{y \in Y} G(x) + \scp{Ax}{y} - F^*(y)
\end{equation}
with the Fenchel conjugate $F^*: Y \to \overline{\RR}$ of $F$. Finding a solution of problem \eqref{eq:saddlepoint} is equivalent to finding a solution to the monotone inclusion~\cite[Th. 19.1]{Bauschke}
\begin{equation}\label{eq:monotone-inclusion}
    0 \in \begin{pmatrix} \partial G & A^* \\ -A & \partial F^* \end{pmatrix}
\end{equation}
In order to find a solution, O'Connor and Vandenberghe proposed in \cite{pddr} to use the Douglas-Rachford method~\cite{Douglas1956, lionsmercier} and apply it to the decomposition
\begin{equation}\label{eq:pddr-splitting}
    0 \in \underbrace{\begin{pmatrix} \partial G & 0 \\ 0 & \partial F^* \end{pmatrix}}_{\cA} \begin{pmatrix} \hat{x} \\ \hat{y} \end{pmatrix} + \underbrace{\begin{pmatrix} 0 & A^* \\ -A & 0 \end{pmatrix}}_{\cB} \begin{pmatrix} \hat{x} \\ \hat{y} \end{pmatrix}.
\end{equation}
With fixed relaxation parameter $\theta \in (0,2)$ and step size $\tau > 0$, this leads to the iteration
\begin{equation}
\label{alg-part:pddr}
\begin{split}
    x^{k+1} &= \prox_{\tau G}\left(p^k\right), \\
    y^{k+1} &= \prox_{\tau F^*}\left(q^k\right), \\
   \begin{bmatrix} v^{k+1} \\ w^{k+1} \end{bmatrix} &= \begin{bmatrix} I & \tau A^* \\ - \tau A & I\end{bmatrix}^{-1} \begin{bmatrix} 2x^{k+1}-p^k \\ 2y^{k+1}-q^k \end{bmatrix}, \\
    p^{k+1} &= p^{k} + \theta \left(v^{k+1} - x^{k+1} \right), \\
    q^{k+1} &= q^{k} + \theta \left(w^{k+1} - y^{k+1} \right).
\end{split}
\end{equation}

However, in practical applications it might happen that the operator and its adjoint are given as two separate implementations of discretizations of a continuous operator and its adjoint, respectively. If the implementations use the ``first dualize, then discretize'' approach, the discretizations might not be adjoint to each other anymore. Sometimes, this is even done on purpose to save computational time or to impose certain structure for the image of the adjoint operator \cite{Buffiere2010InSE, Riddell1995TheAI,zeng2000unmatched}. One example of such an adjoint mismatch is for example provided by MATLAB's \texttt{radon} and \texttt{iradon} functions, with the first function being a discretization of the Radon transform, while the later is a discretization of the adjoint~\cite[Sec. 9.3.5]{CT}. However, both functions are not adjoint to each other and the correct adjoint of \texttt{radon} is not implemented.
In this work we will denote with $A$ the discretization of the forward operator and with $V^{*}$ the discretization of the adjoint and assume that $A^{*}\neq V^{*}$ in general.
The influence of such a mismatch has been studied for the Chambolle-Pock method~\cite{cp_mismatched} and various other algorithms as well \cite{Savanier2021ProximalGA,Chouzenout2021pgm-adjoint,zeng2000unmatched,Lorenz2018TheRK, Dong2019FixingNO, Elfving2018UnmatchedPP, Chouzenoux2023ConvergenceRF}.

If we do not adapt to this adjoint mismatch, the iteration becomes the primal-dual Douglas-Rachford method with mismatched adjoint
\begin{equation}
\label{alg-part:pddr_mismatch}
\begin{split}
    x^{k+1} &= \prox_{\tau G}\left(p^k\right), \\
    y^{k+1} &= \prox_{\tau F^*}\left(q^k\right), \\
   \begin{bmatrix} v^{k+1} \\ w^{k+1} \end{bmatrix} &= \begin{bmatrix} I & \tau V^* \\ - \tau A & I\end{bmatrix}^{-1} \begin{bmatrix} 2x^{k+1}-p^k \\ 2y^{k+1}-q^k \end{bmatrix}, \\
    p^{k+1} &= p^{k} + \theta \left(v^{k+1} - x^{k+1} \right), \\
    q^{k+1} &= q^{k} + \theta \left(w^{k+1} - y^{k+1} \right).
\end{split}
\end{equation}
To get a first impression of its convergence properties, we consider two examples.
\begin{example}
    We consider the minimization problem
    $$\min_{x \in \RR^{128 \times 128}} \tfrac{1}{2}\norm{\cR[x]-b}^2 + \tfrac{0.01}{2} \norm{x}^2,$$
    which has a solution $\hat x$ corresponding to the measurements $b = \cR[x^*]$, where $x^*$ is the famous Shepp-Logan phantom and $\cR$ is the Radon transform. We define $F(y) = \tfrac{1}{2} \norm{y-b}^2$ and $G(x) = \tfrac{0.01}{2} \norm{x}^2$. Accordingly, the Fenchel conjugate of $F^*$ is given by
    $$F^*(y)  = \frac12 \norm{y}^2 + \scp{y}{b}.$$
    The corresponding proximal operators are given by
    \begin{align*}
        \prox_{\tau G}(x) &= \tfrac{x}{1+0.01 \tau}, \\
        \prox_{\tau F^*}(y) &= \tfrac{y-\tau b}{1+\tau}.
    \end{align*}
    We use MATLAB’s implementation of the Radon transform, {\upshape\ttfamily radon}, for the linear operator $A$ and {\upshape\ttfamily iradon} without a filter for the backprojection operator $V^*$ .

    We choose the step size $\tau$ sufficiently small so that a solution to the linear system exists (e.g., $\tau = 0.0008$) and set $\theta = 0.9$. The convergence of the method can be seen in Figure~\ref{fig:pddr-intro-ct}.
    
    \begin{figure}[H]
        \centering
        \includegraphics[width=0.8\linewidth]{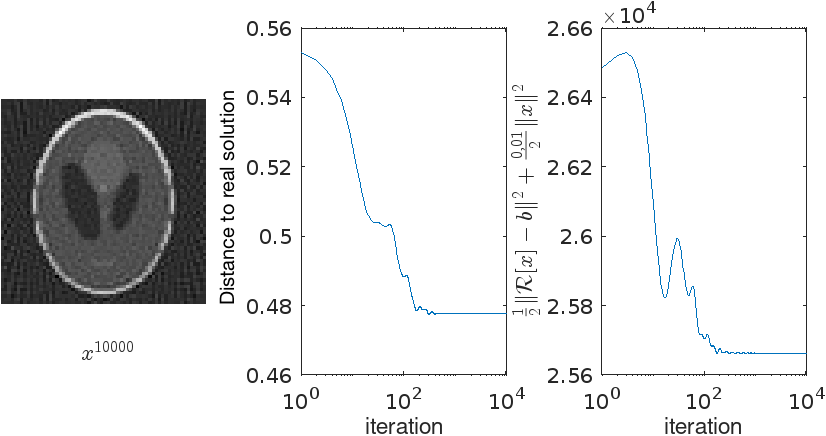}
        \caption{Left: The result after $10,000$ iterations. Middle: The distance between the iterates and the true solution $x^*$. Right: The objective function value over iterations.}
        \label{fig:pddr-intro-ct}
    \end{figure}
The method appears to converge despite the mismatched adjoint, but not to the true solution of the minimization problem.
\end{example}

While the previous example shows a convergent instance of the algorithm, there also exist counterexamples that demonstrate the possible divergence of the method.

\begin{example}\label{pddr:counter-example}  
    We consider the problem $\min_{x \in \RR^d} \norm{x}_1$ and model it using $A = I$, $F(y) = \norm{y}_1$, and $G \equiv 0$. Clearly, the true solution is $x^* = 0$. We examine the case where $V^* = - \alpha I$ for some $\alpha > 0$ as a replacement for $A^* = I$. For sufficiently small $\tau$, the iteration of the primal-dual Douglas-Rachford method with mismatched adjoint for this problem is given by:  
    \begin{align*}  
        x^{k+1} &= (1-\theta) x^{k} + \theta v^{k}, \\  
        y^{k+1} &= \proj_{[-1,1]^d}\left(q^k\right), \\  
        v^{k+1} &= \tfrac{1}{1-\alpha \tau^2} \left(x^{k+1} + \alpha \tau (2 y^{k+1}-q^k)\right), \\  
        w^{k+1} &= \tfrac{1}{1-\alpha \tau^2} \left( \tau x^{k+1} + 2y^{k+1}-q^k\right), \\  
        q^{k+1} &= q^{k} + \theta \left(w^{k+1}-y^{k+1}\right).  
    \end{align*}  
    We set $\alpha = 0.01$, $\tau = 0.1$, and $\theta = 1$. Furthermore, we set the dimension $d = 10$ and choose random standard normally distributed initial vectors. The results are shown in Figure~\ref{fig:pddr-intro-2}.  
    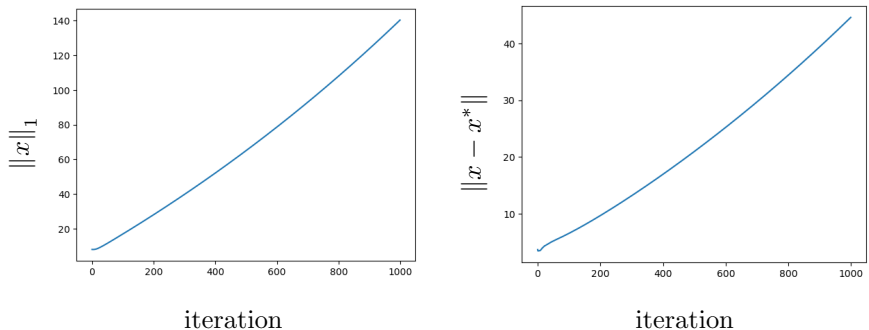
\begin{figure}[H]  
    \centering  
    \begin{tikzpicture}  
        \node at (0, 0) {\includegraphics[width=5cm]{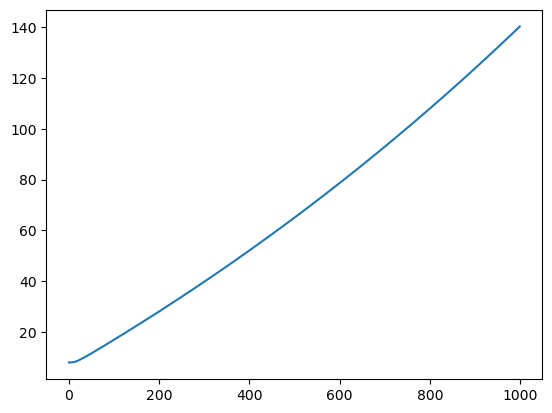}};  
        \node[below] at (0, -2.1) {iteration}; 
        \node[rotate=90] at (-2.8, 0) {$\norm{x}_1$}; 
        \node at (6, 0) {\includegraphics[width=5cm]{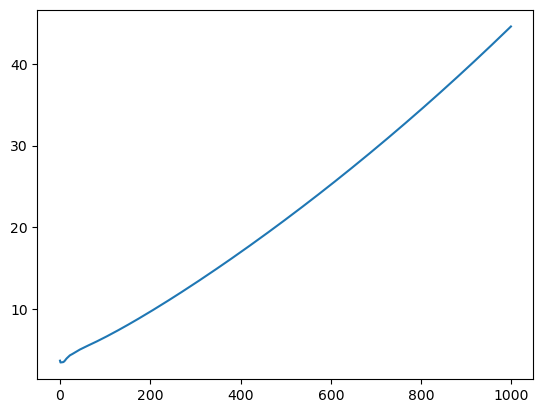}};  
        \node[below] at (6, -2.1) {iteration}; 
        \node[rotate=90] at (3.2, 0) {$\norm{x-x^*}$}; 
    \end{tikzpicture}  
    \caption{Left: Value of the objective function over the iterations. Right: Distance between the iterates and the true solution $x^*$.}  
    \label{fig:pddr-intro-2}  
    \end{figure}  
    It follows that, in this example, the sequence of primal variables $(x^k)_{k \in \NN}$ and thus also $\left(\norm{x^k}_1\right)_{k \in \NN}$ diverges, despite the existence of a fixed point $(\hat x, \hat y, \hat v, \hat w, \hat q) = (0,0,0,0,0)$ for the iteration.  
\end{example} 
These examples show clearly that a more detailed investigation of the convergence properties of algorithm \eqref{alg-part:pddr_mismatch} may be of interest.

The rest of the paper is structured as follows. In Section~\ref{sec:fixedpointproperties} we study the fixed point properties of the method. We show that under mild conditions a fixed point still exists, its unique, and determine an upper bound on the difference between the primal solutions of the primal-dual Douglas-Rachford method with and without an adjoint mismatch. Building on the observations in that section, we introduce an adapted primal-dual Douglas-Rachford method with mismatched adjoint in Section~\ref{sec:adapted-pddrmm}. In Section~\ref{sec:pddrmm} we reformulate the primal-dual Douglas-Rachford method with mismatched adjoint as a preconditioned proximal point algorithm and provide step sizes and mild conditions that lead to the linear convergence of the method. The performance of the algorithm is then demonstrated in Section~\ref{sec:numerical-experiments} on numerical experiments, while the Section~\ref{sec:conclusion} concludes the paper.

\paragraph{Notation}
Throughout this paper, $G: X \to \overline{\RR}$ and $F^*: Y \to \overline{\RR}$ are proper, convex, lower semi-continuous functionals, where $X$ and $Y$ are Hilbert spaces and $\overline{\RR} \defeq \RR \cup \{\infty\}$. With $\norm{\cdot}$ we denote the Hilbert space norm induced by the inner product. Furthermore $A, V \in \mathcal{L}(X,Y)$ are bounded linear operators and $V^*$ is the adjoint of $V$. We will denote the resolvent of a set-valued map $\cA:X\multito X$ as $J_{\cA}:=(I+A)^{-1}$. If $\cA = \partial F$ is the subdifferential of a proper, convex, lower semi-continuous function $F: X \to \overline{\RR}$, we write $J_{\partial F} := \prox_{F}$. Moreover, with $\proj_C$ we denote the orthogonal projection onto a convex set $C$.   Additionally, for set-valued mappings $\cA: X \multito Z$ and $\cB: Y \multito Z$, $x \in X$ and $y \in Y$ and the binary relation $\operatorname{R}$ we write
    $$\cA(x) \operatorname{R} \cB(y) :\Leftrightarrow \forall u \in \cA(x), v \in \cB(y): u \operatorname{R} v.$$

\section{Fixed Point Properties}
\label{sec:fixedpointproperties}
\subsection{Existence and Uniqueness}
While the existence of fixed points of the primal-dual Douglas-Rachford method is guaranteed by existence of a solution to the saddle-point problem~\eqref{eq:saddlepoint}, this is totally different for the method with a mismatched adjoint. Therefore, we begin by characterizing the fixed points.
\begin{lemma}\label{lemma:fixedpoints-pddr}
    A point $(\hat{x}, \hat{y}, \hat{v}, \hat{w}, \hat{p}, \hat{q})$ is a fixed point of the primal-dual Douglas-Rachford method with mismatched adjoint with $\tau \in \left(0,\frac{1}{\norm{A-V}}\right)$ and $\theta \in \RR\setminus \{0\}$ if and only if
    \begin{equation*}
        \hat{v} = \hat{x},\; \hat{w} = \hat{y},\; \hat{p} \in (I+\tau \partial G)\hat{x} \;\text{ and }\; \hat{q} \in (I+\sigma \partial F^*)\hat{y}
    \end{equation*}
    and
    \begin{equation}\label{eq:mismatchedinclusion}
    0 \in  
\begin{pmatrix}
    \partial G & V^* \\ - A & \partial F^*
\end{pmatrix} \begin{pmatrix}
    \hat{x} \\ \hat{y}
\end{pmatrix}.
\end{equation}
\end{lemma}

\begin{proof}
    From the last two lines of the iteration rule, we directly obtain
    \begin{align*}
        \hat p = \hat p + \theta (\hat v - \hat x) &\iff \hat x = \hat v, \\
        \hat q = \hat q + \theta (\hat w - \hat y) &\iff \hat y = \hat w.
    \end{align*}
    Furthermore, by \cite[Prop. 16.44]{Bauschke}, we have
    \begin{align*}
        \hat{x} = \prox_{\tau G}(\hat p) &\iff \hat p - \hat x \in \tau \partial G(\hat x), \\
        \hat{y} = \prox_{\tau F^*}(\hat q) &\iff \hat q - \hat y \in \tau \partial F^*(\hat y).
    \end{align*}
    Now, considering the linear system, we have
    \begin{equation}\label{eq:linear_system}
        \begin{bmatrix} \hat v \\ \hat w \end{bmatrix} = \begin{bmatrix} I & \tau V^* \\ -\tau A & I \end{bmatrix}^{-1} \begin{bmatrix} 2 \hat x - \hat p \\ 2 \hat y - \hat q \end{bmatrix}.
    \end{equation}
    First we discuss if there exists a unique solution to the linear system. With the choice of $$\cB = \begin{bmatrix}
        0 & V^* \\ -A & 0
    \end{bmatrix},$$
     we observe $(I+\tau \cB)$ as a bounded linear operator on $X \times Y$. If $\tau < \tfrac{1}{\norm{\cB}} = \tfrac{1}{\norm{A-V}} $, then it is also injective and, by~\cite[Theorem VI.I.3]{werner_funkana}, surjective. Hence, $(I+\tau \cB)^{-1}$ exists and has full effective domain.
    Furthermore $(\hat v, \hat w)^T$ is a solution of the linear system \eqref{eq:linear_system}
    if and only if
    \begin{equation*}
        \hat v + \tau V^* \hat w = 2 \hat x - \hat p \; \land  \; \hat w - \tau A \hat v = 2 \hat y - \hat q.
    \end{equation*}
    Combining this with the previous observations leads to
    \begin{align*}
        & \hat v + \tau V^* \hat w = 2 \hat x - \hat p \\
        \iff \quad & \hat x + \tau V^* \hat y = 2 \hat x - \hat p \\
        \iff \quad & \tau V^* \hat y = \hat{x} - \hat{p} \in - \tau \partial G(\hat x) \\
        \iff \quad & 0 \in \partial G(\hat x) + V^* \hat y
    \end{align*}
    and
    \begin{align*}
        & \hat w - \tau A \hat v = 2 \hat y - \hat q \\
        \iff \quad & \hat y - \tau A \hat x = 2 \hat y - \hat q \\
        \iff \quad & - \tau A \hat x = \hat{y} - \hat{q} \in - \tau \partial F^*(\hat y) \\
        \iff \quad & 0 \in \partial F^*(\hat y) - A \hat x.
    \end{align*}
    Combining these conditions with
    \begin{align*}
        \hat{x} = \prox_{\tau G}(\hat p) = (I+\tau \partial G)^{-1}(\hat p) &\iff \hat{p} \in (I+\tau \partial G)(\hat x), \\
        \hat{y} = \prox_{\tau F^*}(\hat q) = (I+\tau \partial F^*)^{-1}(\hat q) &\iff \hat{q} \in (I+\tau \partial F^*)(\hat y),
    \end{align*}
    we obtain the desired result.
\end{proof}
Therefore the primal-dual Douglas-Rachford method with mismatched adjoint and the Chambolle-Pock method with mismatched adjoint~\cite{cp_mismatched} solve the same problem. While this result characterizes the fixed points of the method, we have not yet established whether such a solution exists, meaning that the existence of fixed points remains unclear. To address this issue, we will provide conditions that ensure their existence. We will utilize the following result.

\begin{proposition}[{\cite[Prop.~12.54]{rockafellar_variationalanalysis}}]\label{prop:rockafellar-existence-fixed-points}
  If a maximal monotone operator $\cA: X \multito X$ is strongly monotone, then there exists exactly one point $\hat x \in X$ such that $0 \in \cA \hat x$. In fact, $\cA^{-1}x$ is single-valued for all $x \in X$.
\end{proposition}

Applying this proposition to $\cA$ in the zero inclusion~\eqref{eq:mismatchedinclusion} and identifying conditions for strong monotonicity guarantees the existence of unique fixed points. The following theorem provides a suitable result:

\begin{theorem}\label{thm:existenceoffixedpoints}
    Let $G: X \to \overline \RR$ and $F^{*}: Y \to \overline \RR$ be proper, $\gamma_G$- and $\gamma_{F^*}$-strongly convex, lower semicontinuous functions satisfying
    \begin{equation}\label{eq:conditionexistencefixedpoints}
        \gamma_G \gamma_{F^*} > \frac14 \norm{A-V}^2.
    \end{equation}
    Then, there exists a unique solution $\hat{u} \in X \times Y$ satisfying the zero inclusion~\eqref{eq:mismatchedinclusion}.
\end{theorem}
\begin{proof}
    By \cite[Th. 20.25 + Th. 21.2]{Bauschke} and~\cite[Cor. 12.44]{rockafellar_variationalanalysis}, the operator $\cA$ is maximally monotone. According to Proposition~\ref{prop:rockafellar-existence-fixed-points}, we need to show that $\cA$ is strongly monotone. For this, we use \cite[Ex. 22.4]{Bauschke}, which relates the strong convexity of a function to the strong monotonicity of its subgradient. 

    Let
    $$u = \begin{pmatrix} x \\ y \end{pmatrix} \in X \times Y \quad\text{and}\quad u^{\prime} = \begin{pmatrix} x^{\prime} \\ y^{\prime} \end{pmatrix} \in X \times Y.$$
    Then, we have
    \begin{align*}
        \scp{u-u^{\prime}}{\cA u - \cA u^{\prime}} &= \scp{x-x^{\prime}}{\partial G(x) - \partial G(x^{\prime})} - \scp{x-x^{\prime}}{(A-V)(y-y^{\prime})} \\ 
        &\qquad + \scp{y-y^{\prime}}{\partial F^*(y) - \partial F^*(y^{\prime})} \\
        &\geq \gamma_G \norm{x-x^{\prime}}^2 - \scp{x-x^{\prime}}{(A-V)(y-y^{\prime})} + \gamma_{F^*} \norm{y-y^{\prime}}^2.
    \end{align*}
    Using the Cauchy-Schwarz inequality and Young's inequality, we obtain
    \begin{align*}
        \scp{x-x^{\prime}}{(A-V)(y-y^{\prime})} &\leq \norm{A-V} \norm{x-x^{\prime}} \norm{y-y^{\prime}} \\
        &\leq \norm{A-V} \left[\frac{\epsilon}{2} \norm{x-x^{\prime}}^2 + \frac{1}{2 \epsilon} \norm{y-y^{\prime}}^2 \right],
    \end{align*}
    for all $\epsilon > 0$. Substituting this into the previous inequality, we obtain
    \begin{align*}
        \scp{u-u^{\prime}}{\cA u - \cA u^{\prime}} \geq &\left(\gamma_G - \frac{\epsilon}{2} \norm{A-V} \right) \norm{x-x^{\prime}}^2 \\ 
        &\quad +  \left(\gamma_{F^*} - \frac{1}{2 \epsilon} \norm{A-V} \right) \norm{y-y^{\prime}}^2.
    \end{align*}
    The operator $\cA$ is strongly monotone if
    $$\gamma_G - \frac{\epsilon}{2} \norm{A-V} > 0 \quad\text{and}\quad \gamma_{F^*} - \frac{1}{2 \epsilon} \norm{A-V} > 0.$$
    The conditions simplify to
    $$\epsilon < \frac{2 \gamma_G}{\norm{A-V}}\quad \text{and}\quad \epsilon > \frac{1}{2 \gamma_{F^*}} \norm{A-V}.$$
    Such an $\epsilon$ exists if
    $$\frac{1}{2 \gamma_{F^*}} \norm{A-V} < \frac{2 \gamma_G}{\norm{A-V}},$$
    which simplifies to the condition
    $$\gamma_G \gamma_{F^*} > \frac14 \norm{A-V}^2.$$
    This completes the proof.
\end{proof}

To illustrate that the inequality for $A^* \neq V^*$ in condition~\eqref{eq:conditionexistencefixedpoints} must be strictly fulfilled, we consider the following simple example.

\begin{example}\label{ex:existencefixedpoints}
    We examine the saddle point problem
    $$\min_{x \in \RR} \max_{y \in \RR} \frac12 x^2 + (x-3)y - \frac12 y^2,$$
    which originates from the minimization problem
    $$\min_{x \in \RR} \left[\frac12 x^2 + \frac12 (x-3)^2 \right].$$
    We define $G: \RR \to \overline{\RR}, \; G(x) = \tfrac{x^2}{2},$ and $F^*: \RR \to \overline{\RR},\; F^*(y) = \tfrac{y^2}{2} + 3y,$ as well as $A = 1$ and $V^* = -1$.
    
    Both $G$ and $F^*$ are therefore $1$-strongly convex functions, and since $A-V = 2$, we obtain
    $$\gamma_G \gamma_{F^*} = 1 \boldsymbol{=} \frac14 \norm{A-V}^2.$$
    The inequality is thus not strictly satisfied; instead, we have equality. Let us now examine the zero inclusion~\eqref{eq:mismatchedinclusion}.
    
    In our example, with $u = (x,y)^T$, this results in
    \begin{equation*}
       0 \in \cA u = \begin{pmatrix}
            x - y \\ -x + y + 3
        \end{pmatrix} \iff x = y \; \land \; x = y + 3.
    \end{equation*}
    This is a contradiction, showing that $\cA^{-1}(0) = \emptyset$. \vspace{5pt}
    
    However, if we instead set $V^* = -\tfrac12$, then $\norm{A-V}^2 = \left(\tfrac{3}{2}\right)^2 = \tfrac{9}{4} < 4 \gamma_G \gamma_{F^*}$. Consequently, a fixed point exists, and indeed, for $\tilde \cA$, which differs from the previous operator only by replacing $V^*$ with $-\tfrac{1}{2}$, we obtain
    \begin{align*}
       0 \in \tilde \cA u = \begin{pmatrix}
            x - \tfrac{y}{2} \\ -x + y + 3
        \end{pmatrix} &\iff x = \frac{y}{2} \; \land \; x = y + 3 \\
        &\iff x = -3 \;\land \; y = -6.
    \end{align*}
    Thus, a solution exists, but it does not solve the original saddle point problem. The correct solution to the saddle point problem would instead be given by $x = \tfrac{3}{2}$ and $y = - \tfrac{3}{2}$.
\end{example}

Therefore, in general, the existence of solutions to the zero inclusion does not necessarily transfer to cases where a function is not strongly convex, nor to cases where equality holds between the product $\gamma_G \gamma_{F^*}$ and $\tfrac14 \norm{A-V}^2$. 

Before we begin to derive conditions in which the method does converge, we aim to compare the fixed points of the methods with mismatched adjoint to the fixed points of the original methods, i.e., the solutions of the original saddle point problem.

\subsection{Error Estimates}
\label{sec:fehlerabschaetzung}

While the previous section established that fixed points exist under certain conditions, we observed in Example~\ref{ex:existencefixedpoints} that these fixed points do not necessarily coincide with the actual solutions of the underlying saddle point problem. We now aim to analyze this discrepancy and establish an upper bound for it. Intuitively, this bound should depend on the difference between $A^*$ and $V^*$. In particular, the different solutions should be close to each other when the norm $\norm{A-V}$ is small. We will now demonstrate this. 

Starting from the zero inclusions that define the fixed points, we derive the following bound on the difference in the primal variable, which corresponds to an equivalent minimization problem.

\begin{theorem}\label{thm:error-estimate}
    Let $G: X \to \overline \RR$ be a $\gamma_G$-strongly convex function and let
    \begin{equation*}
         0 \in 
\begin{pmatrix}
    \partial G & A^* \\ - A & \partial F^*
\end{pmatrix} \begin{pmatrix}
    x^* \\ y^*
\end{pmatrix} \quad\text{and}\quad
 0 \in 
\begin{pmatrix}
    \partial G & V^* \\ - A & \partial F^*
\end{pmatrix} \begin{pmatrix}
    \hat{x} \\ \hat{y}
\end{pmatrix}.
    \end{equation*}
    Then the following estimate holds:
\[
  \norm{x^* - \hat{x}} \leq \frac{1}{\gamma_G} \norm{(V-A)^* \hat{y}}.
\]
\end{theorem}

\begin{proof}
See \cite[Th. 1.2]{cp_mismatched}.
\end{proof}

This bound is even sharp, as we will demonstrate with the following example.
\begin{example}
    We consider the saddle point problem
    $$\min_{x \in \RR} \max_{y \in \RR} \frac14 x^2  + xy - |y|$$
    with $G: \RR \to \overline{\RR}, \; G(x) = \tfrac{x^2}{4},$ and $F^*: \RR \to \overline{\RR},\; F^*(y) = |y|,$ as well as $A = 1$. 
    
    The solution to this problem is given by the pair $\left(x^*, y^*\right) = \left(0,0\right)$, and the primal-dual Douglas-Rachford method converges to this solution when the step sizes are chosen appropriately.
    
    If we now replace the adjoint of $A$ with $V^* = \tfrac12$, the fixed points are given by $\{\left(\hat x, - \hat x\right) \mid \hat x \in \{0,1,-1\}\}.$ In this example, we have $\gamma_G = \tfrac12$ and $A-V = \tfrac12$. Examining the difference in the first (primal) component of the fixed points, we obtain for $\left(\hat x,\hat y\right) = (0,0)$:
    $$0 = \norm{x^*-\hat x} = \frac{1}{\gamma_G} \norm{(A-V)^* \hat{y}} =  2 \left\|\frac12 \cdot 0\right\| = 0.$$
    Similarly, for the other fixed points $\left(\hat x,\hat y\right) \in \{(1,-1),(-1,1)\}$, we also have equality:
    $$1 = \norm{x^*-\hat{x}} =  \frac{1}{\gamma_G} \norm{(A-V)^* \hat{y}} = 1.$$
    \end{example}

In numerical experiments, the upper bound often provides a good approximation of the actual distance between the fixed points of \eqref{alg-part:pddr_mismatch} and the true solution. This is illustrated, for example, in Figure~\ref{fig:pddrmm-quadratic_linear_conv}.

\section{The Adapted Primal-Dual Douglas-Rachford Method with Mismatched Adjoint}  
\label{sec:adapted-pddrmm}  

While we want to observe the unchanged method except for the substitution of the adjoint, it is possible to construct an adapted version of the primal-dual Douglas-Rachford method with mismatched adjoint under a few mild assumptions. 

For the formulation of this adapted version, we return to the definition of the Douglas-Rachford method. Recall that this method is based on decomposing the monotone inclusion as  
\begin{equation*}
    0 \in \cA u + \cB u
\end{equation*}  
with maximally monotone operators \(\cA, \cB: X \multito X\), and seeks a solution \(u^*\). For the formulation of the primal-dual Douglas-Rachford method with mismatched adjoint, this corresponds to solving the decomposition
\begin{equation}\label{eq:pddrmm-splitting}
    0 \in \underbrace{\begin{pmatrix} \partial G & 0 \\ 0 & \partial F^* \end{pmatrix}}_{\cA} \begin{pmatrix} \hat{x} \\ \hat{y} \end{pmatrix} + \underbrace{\begin{pmatrix} 0 & V^* \\ -A & 0 \end{pmatrix}}_{\cB} \begin{pmatrix} \hat{x} \\ \hat{y} \end{pmatrix}.
\end{equation}
As already noted, \(\cB\) is not monotone when \(A^* \neq V^*\). However, in the convergent example~\ref{fig:pddr-intro-ct}, both \(G\) and \(F^*\) are strongly convex, making their subgradients strongly monotone. In this case, we can ``transfer'' some monotonicity of \(\cA\) to \(\cB\). Using the decomposition  
\[
0 \in \begin{pmatrix}
    \partial G & V^* \\ -A & \partial F^*
\end{pmatrix}
\begin{pmatrix}
    \hat x \\ \hat y
\end{pmatrix} = \underbrace{\begin{pmatrix}
    \partial G - \mu_G I & 0 \\ 0 & \partial F^* -\mu_{F^*} I
\end{pmatrix}}_{\Tilde \cA}
\begin{pmatrix}
    \hat x \\ \hat y
\end{pmatrix} + \underbrace{\begin{pmatrix}
     \mu_G I & V^* \\ -A & \mu_{F^*} I
\end{pmatrix}}_{\Tilde \cB}
\begin{pmatrix}
    \hat x \\ \hat y
\end{pmatrix}
\]
with \(0 < \mu_G \leq \gamma_G\) and \(0 < \mu_{F^*} \leq \gamma_{F^*}\), satisfying  
\[
\mu_G \mu_{F^*} \geq \frac14 \norm{A-V}^2,
\]
both \(\Tilde \cA\) and \(\Tilde \cB\) are maximally monotone. Consequently, the iteration changes to  
\begin{equation}\label{alg:adapted-pddr}
\begin{split}
    x^{k+1} &= \prox_{\tau \left(G-\tfrac{\mu_G}{2}\norm{\cdot}^2\right)}(p^k), \\
    y^{k+1} &= \prox_{\tau \left(F^*-\tfrac{\mu_{F^*}}{2}\norm{\cdot}^2\right)}(q^k), \\
    \begin{bmatrix} v^{k+1} \\ w^{k+1} \end{bmatrix} &= \begin{bmatrix}
        (1+\tau \mu_G) I & \tau V^* \\ -\tau A & (1+\tau \mu_{F^*}) I
    \end{bmatrix}^{-1} \begin{bmatrix}
        2x^{k+1} - p^k \\ 2y^{k+1}-q^k
    \end{bmatrix}, \\
    p^{k+1} &= p^k + \theta \left(v^{k+1} - x^{k+1}\right), \\
    q^{k+1} &= q^k + \theta \left(w^{k+1} - y^{k+1}\right).
\end{split}
\end{equation}
If \(\tau < \min\left\{ \tfrac{1}{\mu_G}, \tfrac{1}{\mu_{F^*}}\right\}\), the updates of \(x^{k+1}\) and \(y^{k+1}\) can be rewritten using 
\begin{align*}
    p = \prox_{\lambda \left(F - \tfrac{\mu}{2} \norm{\cdot}^2\right)}(x) &\iff\quad x-p \in \lambda \partial F(p) - \lambda \mu p \\
    &\overset{\text{\makebox[0pt]{if $\lambda < \mu^{-1}$}}}{\iff}\quad \tfrac{1}{1-\lambda \mu} x - p \in \tfrac{\lambda}{1-\lambda \mu} \partial F(p) \\
    &\iff\quad p = \prox_{\tfrac{\lambda}{1-\lambda \mu} F}\left(\tfrac{1}{1-\lambda \mu} x\right)
\end{align*}
for proper, \(\mu\)-strongly convex, lower semicontinuous functions \(F: X \to \overline{\RR}\) and positive \(\lambda\), through the proximal operators of \(G\) and \(F^*\) as  
\begin{align*}
    x^{k+1} &= \prox_{\tfrac{\tau}{1-\tau \mu_G} G}\left(\tfrac{p^k}{1-\tau \mu_G}\right), \\
    y^{k+1} &= \prox_{\tfrac{\tau}{1-\tau \mu_{F^*}} F^*}\left(\tfrac{q^k}{1-\tau \mu_{F^*}}\right).
\end{align*}
The convergence of the method then directly follow from those of the Douglas-Rachford method. The existence of a fixed point is guaranteed by Theorem~\ref{thm:existenceoffixedpoints}.

\section{The Primal-Dual Douglas-Rachford Method with an Adjoint Mismatch}
\label{sec:pddrmm}

We come back now the the unchanged method \eqref{alg-part:pddr_mismatch}. To prove the convergence of the primal-dual Douglas-Rachford method with mismatched adjoint, we represent it as a relaxed preconditioned proximal point method, similar to the Peaceman-Rachford method in~\cite{naldi}. Consequently, the proof is based on the techniques introduced in that paper. We structure this section as follows: first, we describe the zero inclusion differently and show that the corresponding relaxed preconditioned proximal point method is equivalent to the primal-dual Douglas-Rachford method with mismatched adjoint; and finally, we prove that it converges linearly under certain conditions.  

\subsection{Preconditioned Proximal Point Interpretation}  
We begin by establishing an interpretation of the primal-dual Douglas-Rachford method with mismatched adjoint as an equivalent relaxed preconditioned proximal point method. We use the formulation from~\cite[Eq. (3.7)]{naldi}. While this formulation was originally derived for monotone operators, monotonicity is not essential at this stage. To ensure its validity and given its significant role in the following proofs, we will nonetheless provide a proof of this formulation.  

\begin{lemma}  
   The condition $0 \in \mathcal{A} u + \mathcal{B} u$ holds if and only if there exists a vector $\textbf{u}$ whose first component is $u$ such that  
   \begin{equation}\label{eq:desplittedzeroinclusion}  
   0 \in \underbrace{\left[\begin{array}{ccc} \alpha I + \gamma \mathcal{A} & -I & -I \\ I & 0 & -I \\ (1-2\alpha) I & I & \alpha I + \gamma \mathcal{B} \end{array}\right]}_{=: \mathbb{A}_{\alpha}} \textbf{u} 
   \end{equation}  
   for any choice of $\gamma \in \CC$ and $\alpha \in \CC$.
\end{lemma}  

\begin{proof} 
   "$\Rightarrow$": \; If $0 \in \mathcal{A}u + \mathcal{B} u$, $\gamma \in \CC$, and $\alpha \in \CC$, then there exist $a \in \gamma \mathcal{A} u$ and $b \in \gamma \mathcal{B} u$ such that $$0 = a + b.$$  
   Therefore, defining $$\textbf{u} = \begin{pmatrix} u \\ (\alpha-1) u + a \\ u \end{pmatrix},$$  
   we obtain  
   \begin{align*}  
   \mathbb{A}_{\alpha} \textbf{u} &= \left[ \begin{array}{c} \alpha u + \gamma \mathcal{A} u - (\alpha-1) u - a - u \\ u - u \\ (1-2\alpha) u + (\alpha-1) u + a + (\alpha I + \gamma \mathcal{B}) u \end{array} \right]  \\  
   &=\left[ \begin{array}{c} \gamma \mathcal{A} u - a \\ 0 \\ a + \gamma \mathcal{B} u\end{array} \right] = \left[ \begin{array}{c} \gamma \mathcal{A} u - a \\ 0 \\ -b + \gamma \mathcal{B} u \end{array} \right] \ni 0.  
   \end{align*}  

   \noindent "$\Leftarrow$": \; For the reverse direction, if  
   $$0 \in \mathbb{A}_{\alpha} \begin{pmatrix} u \\ r \\ s \end{pmatrix} = \left[ \begin{array}{c} \alpha u + \gamma \mathcal{A} u - r - s \\ u - s \\ (1-2\alpha) u + r + \alpha s + \gamma \mathcal{B} s \end{array}\right],$$  
   then it follows that $u = s$ and  
   \begin{align*}  
   	\gamma \mathcal{A} u &\ni u + r - \alpha u = (1-\alpha) u + r, \\  
  	\gamma \mathcal{B} u &\ni (2\alpha - 1) u - r - \alpha u = (\alpha - 1) u - r,  
   \end{align*}  
   hence, we conclude that $$0 \in \mathcal{A} u + \mathcal{B} u.$$  
\end{proof}  

Thus, we have expressed the zero inclusion~\eqref{eq:mismatchedinclusion} to be solved in an equivalent form in a lifted space. In the next step, we now consider the corresponding relaxed preconditioned proximal point method,
\begin{align*}
\mathbf{u}^{k+1} \in (M+\mathbb{A}_{\alpha})M\mathbf{u}^{k}
\end{align*}
for some linear, self-adjoint and positive semi-definite map $M$ that is called preconditioner.

As a preconditioner, we use, as in~\cite[p. 14]{naldi},
\begin{equation}\label{eq:rd-def-M}
    M = \begin{bmatrix}
        I & I & I \\ I & I & I \\ I & I & I
    \end{bmatrix} \in \mathcal{L}\left((X \times Y)^3, (X \times Y)^3\right).
\end{equation}
To establish convergence using Theorems~\cite[Th. 2.14]{naldi} and~\cite[Th. 2.19]{naldi}, we must first prove that $M$ is indeed an admissible preconditioner for $\mathbb{A}_{\alpha}$ according to~\cite[Def. 2.1 + Eq. (2.3)]{naldi}.

\begin{lemma}\label{lm:degeneratepreconditioner}
    Let $\gamma > 0$, $\alpha > -1$, and $\mathbb{A}_{\alpha}$ be defined as in~\eqref{eq:desplittedzeroinclusion}. If 
    \begin{equation}\label{eq:def-tau-gamma}
    \tau \defeq \tfrac{\gamma}{1+\alpha} < \tfrac{1}{\norm{A-V}},
    \end{equation}
    then $M$ in~\eqref{eq:rd-def-M} is an admissible preconditioner for $\mathbb{A}_{\alpha}$ in the sense of \cite[Def. 2.1]{naldi}.
\end{lemma}

\begin{proof}
    It is clear that $M \in \cL \left((X \times Y)^3, (X \times Y)^3\right)$, it is self-adjoint and positive semi-definite. Thus, it remains to verify that $\mathcal{T} \defeq (M+\mathbb{A}_{\alpha})^{-1} M$ is single-valued and that $\dom \mathcal{T} = (X \times Y)^3$.

    We follow the approach in~\cite[p. 14]{naldi}, starting with $u \in (X \times Y)^3$ and
    \begin{align*}
        & v \in (M+\mathbb{A}_{\alpha})^{-1} M u \\
        \iff\quad& (M+\mathbb{A}_{\alpha}) v \ni Mu.
    \end{align*}
    Considering the right-hand side, we have
    $$M u = \begin{bmatrix}
        I & I & I \\ I & I & I \\ I & I & I
    \end{bmatrix} u = \begin{bmatrix}
        \bar u \\ \bar u \\ \bar u
    \end{bmatrix}$$
    for some $\bar u \in X \times Y$.
    
    Using the definition of $\mathbb{A}_{\alpha}$, along with $\tau = \tfrac{\gamma}{1+\alpha}$, and writing $v = (v_1,v_2,v_3)^T$ with $v_1,v_2,v_3 \in X \times Y$, we obtain the sequence of equivalences
    \begin{align*}
        & Mu \in (M+\mathbb{A}_{\alpha}) v \\
        \iff\quad&\left\{\begin{array}{rl}
            \left((\alpha+1)I+\gamma \cA\right)v_1 &\ni \bar u, \\
            2v_1+v_2 &= \bar u, \\
            2(1-\alpha)v_1+2v_2+\left((\alpha+1)I+\gamma \cB\right)v_3 &\ni \bar u.
        \end{array}\right. \\
        \iff\quad&\left\{\begin{array}{rl}
            v_1 &= J_{\tau \cA} \left(\tfrac{\bar u}{\alpha+1}\right), \\
            v_2 &= \bar u - 2 J_{\tau \cA} \left(\tfrac{\bar u}{\alpha+1}\right), \\
            \left(I+\tau \cB\right)v_3 &\ni 2  J_{\tau \cA} \left(\tfrac{\bar u}{\alpha+1}\right) - \tfrac{\bar u}{\alpha+1}.
        \end{array}\right.
    \end{align*}

    Now, since $$\cB = \begin{pmatrix}
        0 & V^* \\ -A & 0
    \end{pmatrix}$$
    is not monotone, its resolvent is not defined. However, as in the proof of Lemma \ref{lemma:fixedpoints-pddr} the operator  $(I+\tau \cB)$ is a bounded linear operator on $X \times Y$. If $\tau < \tfrac{1}{\norm{\cB}} = \tfrac{1}{\norm{A-V}} $, then it is also injective and, by~\cite[Theorem VI.I.3]{werner_funkana}, surjective. Hence, $J_{\tau \cB}=(I+\tau \cB)^{-1}$ exists and has full effective domain. Thus, we obtain
    \begin{align*}
        & Mu \in (M+\mathbb{A}_{\alpha}) v \\
        \iff\quad&\left\{\begin{array}{rl}
            v_1 &= J_{\tau \cA} \left(\tfrac{\bar u}{\alpha+1}\right), \\
            v_2 &= \bar u - 2 J_{\tau \cA} \left(\tfrac{\bar u}{\alpha+1}\right), \\
            v_3 &= J_{\tau \cB} \left(2  J_{\tau \cA} \left(\tfrac{\bar u}{\alpha+1}\right) - \tfrac{\bar u}{\alpha+1} \right).
        \end{array}\right.
    \end{align*}

    Using the properties of the resolvent, we conclude that $(M+\mathbb{A}_{\alpha})^{-1} M$ is single-valued and has full effective domain $\dom \mathcal{T} = (X \times Y)^3$. Therefore, $M$ is a admissible preconditioner for $\mathcal{A}_{\alpha}$.
\end{proof}

Having established that $M$ is a admissible preconditioner for $\mathcal{A}_{\alpha}$ when $\tfrac{\gamma}{1+\alpha} < \tfrac{1}{\norm{A-V}}$, we also observe the identity
\begin{equation}
    M = CC^* \text{ with } C = \begin{bmatrix}
    I \\ I \\ I
\end{bmatrix} \in \cL\left(X \times Y, (X \times Y)^3 \right).
\end{equation}

Thus, we now consider the iteration
\begin{equation}
    w^{k+1} = w^k + \lambda_k \left(C^*(M+\mathbb{A}_{\alpha})^{-1}C w^k - w^k\right) \label{eq:degeneratedpreconditionedproximalpoint}
\end{equation}
from \cite[Eq. (2.15)]{naldi}.
We aim to show that this is equivalent to the primal-dual Douglas-Rachford method with mismatched adjoint~\eqref{alg-part:pddr_mismatch}.

\begin{theorem}\label{thm:equivalence_pddr_precond}
    Let $\cA$ and $\cB$ be defined as in~\eqref{eq:pddrmm-splitting}. Furthermore, let $\gamma > 0$ and $\alpha > -1$ be chosen such that
    \begin{equation*}
        \tau \defeq \tfrac{\gamma}{1+\alpha} < \tfrac{1}{\norm{A-V}}
    \end{equation*}
    holds, and let $\mathbb{A}_{\alpha}$ be defined as in~\eqref{eq:desplittedzeroinclusion}. Then, the iteration~\eqref{eq:degeneratedpreconditionedproximalpoint} with relaxation parameters $\lambda_k = \lambda \in (0,2)$ for all $k \in \NN$ is equivalent to the primal-dual Douglas-Rachford method with mismatched adjoint~\eqref{alg-part:pddr_mismatch} with extrapolation parameter $\theta = \tfrac{\lambda}{1+\alpha}$ in the sense that the relation
    \begin{align*}
        \begin{bmatrix} x^{k+1} \\ y^{k+1} \end{bmatrix} = J_{\tau \cA}\left(\Tilde{w}^k \right), \; 
        \begin{bmatrix} v^{k+1} \\ w^{k+1} \end{bmatrix} = J_{\tau \cB}\left(2 J_{\tau \cA}\left(\Tilde{w}^k \right) - \Tilde{w}^k\right), \;
        \begin{bmatrix} p^{k+1} \\ q^{k+1} \end{bmatrix} = \Tilde{w}^{k+1}
    \end{align*}
    with $\Tilde{w}^{k} \defeq \tfrac{w^k}{1+\alpha}$ holds for the iterates of both algorithms for all $k \in \NN$.
\end{theorem}

\begin{proof}
    From the considerations in the proof of the previous Lemma~\ref{lm:degeneratepreconditioner}, we have
    \begin{equation*}
        C^*(M+\mathbb{A}_{\alpha})^{-1}C w^k = w^k - J_{\tau \cA}\left(\tfrac{w^k}{1+\alpha}\right) + J_{\tau \cB}\left(2 J_{\tau \cA}\left(\tfrac{w^k}{1+\alpha}\right) - \tfrac{w^k}{1+\alpha}\right).
    \end{equation*}
    Thus, we obtain
    \begin{align*}
        w^{k+1} &= w^k + \lambda \left(C^*(M+\mathbb{A}_{\alpha})^{-1}C w^k - w^k\right) \\
        &= w^{k} + \lambda \left(J_{\tau \cB}\left(2 J_{\tau \cA}\left(\tfrac{w^k}{1+\alpha}\right) - \tfrac{w^k}{1+\alpha}\right) - J_{\tau \cA}\left(\tfrac{w^k}{1+\alpha}\right) \right).
    \end{align*}
    Now, considering $\Tilde{w}^{k} = \tfrac{w^k}{1+\alpha}$ for $k \in \NN$, we obtain
    \begin{equation*}
        \Tilde{w}^{k+1} = \Tilde{w}^k + \theta \left( J_{\tau \cB}\left(2 J_{\tau \cA}\left(\Tilde{w}^k \right) - \Tilde{w}^k \right) - J_{\tau \cA}\left(\Tilde{w}^k \right) \right),
    \end{equation*}
    which corresponds exactly to the primal-dual Douglas-Rachford method with mismatched adjoint~\eqref{alg-part:pddr_mismatch}. Here, we have
    \begin{align*}
        \begin{bmatrix} x^{k+1} \\ y^{k+1} \end{bmatrix} = J_{\tau \cA}\left(\Tilde{w}^k \right), \; 
        \begin{bmatrix} v^{k+1} \\ w^{k+1} \end{bmatrix} = J_{\tau \cB}\left(2 J_{\tau \cA}\left(\Tilde{w}^k \right) - \Tilde{w}^k\right), \;
        \begin{bmatrix} p^{k+1} \\ q^{k+1} \end{bmatrix} = \Tilde{w}^{k+1}.
    \end{align*}
    Therefore, the two methods are identical up to a scaling of the iterates.
\end{proof}

From the convergence of the equivalent relaxed preconditioned proximal point method in the theorem, we also conclude the convergence of the primal-dual Douglas-Rachford method with mismatched adjoint. In particular, due to the non-expansiveness of the resolvent from \cite[Prop. 23.8]{Bauschke}, the primal variables $\left(x^{k+1}\right)_{k\in\NN}$ and the dual variables $\left(y^{k+1}\right)_{k\in\NN}$ also converge.

\subsection{Linear Convergence}

If we can show that the relaxed preconditioned proximal point method from Theorem~\ref{thm:equivalence_pddr_precond} converges linearly, it follows that the iterates of the primal-dual Douglas-Rachford method with mismatched adjoint~\eqref{alg-part:pddr_mismatch} also exhibit linear convergence. To establish this, we will use~\cite[Th. 2.19]{naldi} and \cite[Prop. 2.5 + Prop. 2.18]{naldi}.

We begin our analysis by identifying conditions for the monotonicity of $\mathbb{A}_{\alpha}$. In the introduction, we observed that in the case of strongly convex $G$ and $F^*$, an alternative decomposition of the zero inclusion~\eqref{eq:mismatchedinclusion} leads to a method that still maintains linear convergence. The idea behind this decomposition was to introduce operators $\Sigma_{\cA} \in \cL(X\times Y, X \times Y)$ and $\Sigma_{\cB} \in \cL(X\times Y, X \times Y)$ such that both $\cA - \Sigma_{\cA}$ and $\cB + \Sigma_{\cB}$ remain maximally monotone. To ensure that the solution remains unchanged, the adapted primal-dual Douglas-Rachford method with mismatched adjoint required $\Sigma_{\cA} = \Sigma_{\cB}$. We now adopt this idea to prove the convergence of the non-adapted method.

Accordingly, we write
\begin{equation*}
    \mathbb{A}_{\alpha} = \begin{bmatrix}
        \alpha I + \gamma \Sigma_{\cA} + \gamma \left(\mathcal{A}-\Sigma_{\cA}\right) & -I & -I \\ 
        I & 0 & -I \\ (1-2\alpha) I & I & 
        \alpha I - \gamma \Sigma_{\cB} + \gamma \left(\mathcal{B}+\Sigma_{\cB}\right)
    \end{bmatrix}.
\end{equation*}

If $\cA - \Sigma_{\cA}$ and $\cB + \Sigma_{\cB}$ are maximally monotone operators, then $\mathbb{A}_{\alpha}$ is monotone for $\gamma > 0$ if
\begin{equation}\label{eq:equivalent-monotone-operator}
    \Tilde{\mathbb{A}}_{\alpha} \defeq \begin{bmatrix}
        \alpha I + \gamma \Sigma_{\cA} & 0 & 0 \\
        0 & 0 & 0 \\
        - 2 \alpha I & 0 & \alpha I  - \gamma \Sigma_{\cB}
    \end{bmatrix}
\end{equation}
is monotone. We set
\begin{equation}\label{eq:pddr-sigmas}
    \Sigma_{\cA} = \begin{pmatrix} \mu_G I & 0 \\ 0 & \mu_{F^*} I \end{pmatrix} \; \text{ and }\; \Sigma_{\cB} = \begin{pmatrix} \Tilde{\mu}_G I & 0 \\ 0 & \Tilde{\mu}_{F^*} I \end{pmatrix}.
\end{equation}
By appropriately choosing $\Sigma_{\cA}$ and $\Sigma_{\cB}$, we can derive conditions that guarantee monotonicity.

\begin{lemma}\label{lm:conds-alpha-monotony}
    Let $G$ and $F^*$ be strongly convex functions with constants $\gamma_G$ and $\gamma_{F^*}$, respectively, and let $\mu_G, \mu_{F^*}, \Tilde{\mu}_G$, and $\Tilde{\mu}_{F^*}$ be constants satisfying 
    $$0 < \Tilde{\mu}_G < \mu_G < \gamma_G, \quad 0 < \Tilde{\mu}_{F^*} < \mu_{F^*} < \gamma_{F^*},$$
    and 
    $$\Tilde{\mu}_G \Tilde{\mu}_{F^*} \geq \frac14 \norm{A-V}^2.$$
    Then, $\mathbb{A}_{\alpha}$ is monotone if 
    \begin{equation}\label{eq:conds-alpha}
    \alpha \geq \max\left\{\tfrac{\gamma \mu_G \Tilde{\mu}_G}{\mu_G - \Tilde{\mu}_G}, \tfrac{\gamma \mu_{F^*} \Tilde{\mu}_{F^*}}{\mu_{F^*} - \Tilde{\mu}_{F^*}} \right\}
    \end{equation}
    holds.
\end{lemma}

\begin{proof}
    Let $\Sigma_{\cA}$ and $\Sigma_{\cB}$ be defined as in Equation~\eqref{eq:pddr-sigmas}. 

    With the chosen values of $\mu_G, \mu_{F^*}, \Tilde{\mu}_G$, and $\Tilde{\mu}_{F^*}$, both $\cA - \Sigma_{\cA}$ and $\cB + \Sigma_{\cB}$ are maximally monotone. 

    Based on the previous analysis, we now investigate the monotonicity of the operator $\Tilde{\mathbb{A}}_{\alpha}$. Substituting the definitions of $\Sigma_{\cA}$ and $\Sigma_{\cB}$, we obtain the operator
    \begin{equation*}
        \Tilde{\mathbb{A}}_{\alpha} = \begin{bmatrix}
            (\alpha + \gamma \mu_G) I & 0 & 0 & 0 & 0 \\
            0 & (\alpha + \gamma \mu_{F^*}) I & 0 & 0 & 0 \\
            0 & 0 & 0 & 0 & 0 \\
            -2 \alpha I & 0 & 0 & (\alpha - \gamma \Tilde{\mu}_G) I & 0 \\
            0 & -2 \alpha I & 0 & 0 & (\alpha - \gamma \Tilde{\mu}_{F^*}) I 
        \end{bmatrix}.
    \end{equation*}This operator maps from $X \times Y \times (X \times Y) \times X \times Y$ into itself as a bounded linear operator.

    To analyze its monotonicity, we consider a vector $v = (a,b,c,d,e)$ with $a,d\in X, b,e\in Y$ and $c\in X\times Y$ in this space and compute
    \begin{align*}
        \scp{v}{\Tilde{\mathbb{A}}_{\alpha} v} =&\; (\alpha + \gamma \mu_G) \norm{a}^2 - 2 \alpha \scp{a}{d} + (\alpha - \gamma \Tilde{\mu}_{G}) \norm{d}^2 \\
        &\; + (\alpha + \gamma \mu_{F^*}) \norm{b}^2 - 2 \alpha \scp{b}{e} + (\alpha - \gamma \Tilde{\mu}_{F^*}) \norm{e}^2.
    \end{align*}
    
    We quickly see that both summands follow the general form
    $$(\alpha + \gamma \mu) \norm{x}^2 - 2 \alpha \scp{x}{y} + (\alpha - \gamma \Tilde{\mu}) \norm{y}^2$$
    for specific choices of $x, y, \mu$, and $\Tilde{\mu}$. 

    Now, applying the Young's inequality, we obtain that
    $$(\alpha + \gamma \mu) \norm{x}^2 - 2 \alpha \scp{x}{y} + (\alpha - \gamma \Tilde{\mu}) \norm{y}^2 \geq 0$$
    holds if $\alpha > \gamma \Tilde{\mu}$ and $(\alpha + \gamma \mu)(\alpha - \gamma \Tilde{\mu}) \geq \alpha^2$ are satisfied. This leads to the condition
    $$\alpha \geq \tfrac{\gamma \mu \Tilde{\mu}}{\mu - \Tilde{\mu}}.$$
    
    Moreover, this condition also includes 
    $$\alpha \geq \tfrac{\gamma \mu \Tilde{\mu}}{\mu - \Tilde{\mu}} = \tfrac{\gamma \Tilde{\mu}}{\tfrac{\mu - \Tilde{\mu}}{\mu}} = \tfrac{\gamma \Tilde{\mu}}{1-\tfrac{\Tilde{\mu}}{\mu}} > \gamma \Tilde{\mu}.$$
    
    Thus, $\mathbb{A}_{\alpha}$ is monotone if
    $$\alpha \geq \max\left\{\tfrac{\gamma \mu_G \Tilde{\mu}_G}{\mu_G - \Tilde{\mu}_G}, \tfrac{\gamma \mu_{F^*} \Tilde{\mu}_{F^*}}{\mu_{F^*} - \Tilde{\mu}_{F^*}} \right\} > 0.$$
\end{proof}

At this point, we also already know that choosing $\theta \in \left(0, \tfrac{2}{\alpha+1}\right)$ with an $\alpha$ satisfying condition~\eqref{eq:conds-alpha} ensures the convergence of the primal-dual Douglas-Rachford method with mismatched adjoint to a fixed point, provided that such a fixed point exists.

\begin{corollary}\label{cor:weak_convergence}
    Let the conditions from Lemma~\ref{lm:conds-alpha-monotony} be satisfied and let
    $$c = \min\left\{\tfrac{\mu_G - \Tilde{\mu}_G}{\mu_G \Tilde{\mu}_G},\tfrac{\mu_{F^*} - \Tilde{\mu}_{F^*}}{\mu_{F^*} \Tilde{\mu}_{F^*}}\right\}.$$
    Then, the primal-dual Douglas-Rachford method with mismatched adjoint converges weakly for any 
    \begin{equation*}        
    \tau \in \left(0,\min(\tfrac{1}{\norm{A-V}},c)\right)\,\text{ and }\,\theta \in \left(0, 2-2 \tfrac{\tau}{c}\right)
    \end{equation*}
    to the fixed point $\hat u \in \left(\cA + \cB\right)^{-1}(0)$.
\end{corollary}

\begin{proof}
    By theorem \ref{thm:equivalence_pddr_precond} the iteration \eqref{eq:degeneratedpreconditionedproximalpoint} is equivalent to the primal-dual Douglas-Rachford method with mismatched adjoint, whenever $0 < \tau = \tfrac{\gamma}{1+\alpha} < \norm{A-V}^{-1}$ and $\theta = \tfrac{\lambda}{1+\alpha}$ with $\lambda \in (0,2)$ are satisfied for some $\alpha > -1.$
    Inserting the definition of $\tau$ into \eqref{eq:conds-alpha} yields,
    if the conditions from Lemma~\ref{lm:conds-alpha-monotony} are satisfied, that
    $$\tau \leq \tfrac{\alpha}{1+\alpha} c$$
    has to be fulfilled as well. Therefore $\tfrac{\alpha}{1+\alpha} \in (0,1)$ for $\alpha > 0$ yields to the upper bound on $\tau$. Furthermore it is $\tau \leq \tfrac{\alpha}{1+\alpha} c \Leftrightarrow \alpha \geq \tfrac{\tau}{c-\tau}$ and hence $0 < \theta < \frac{2}{1+\alpha} < \tfrac{2}{1+\tfrac{\tau}{c-\tau}} = 2 - 2 \tfrac{\tau}{c}$.
    Theorem \ref{thm:existenceoffixedpoints} guarantees the existence of an unique fixed point.
    
    The convergence now follows directly from~\cite[Th. 2.14]{naldi} with \cite[Prop. 2.5]{naldi} and the previous results.
\end{proof}

In fact, for certain step sizes, the iterates exhibit linear convergence. Before proving this, we first introduce a few useful statements.

\begin{lemma}\label{lemma:parallelogramm}
    Let $(\mathcal{U}, \norm{\cdot})$ be a real inner product space. Then, for any $a,b,c \in \mathcal{U}$, the following identity holds:
    $$\left\|a+b+c\right\|^2-\left\|a+b-c\right\|^2=\left\|a+c\right\|^2+\left\|b+c\right\|^2-\left\|a-c\right\|^2-\left\|b-c\right\|^2.$$
\end{lemma}

\begin{proof}
Using the parallelogram identity, we compute:
\begin{align*}
&\; \| a +b+c\left\|^2-\right\| a+b-c\left\|^2-\right\| a+c\left\|^2-\right\| b+c\left\|^2+\right\| a-c\left\|^2+\right\| b-c \|^2 \\
= &\, \left\|a+b+c\right\|^2-\left\|a+b-c\right\|^2-\left\|a+c\right\|^2-\left\|b\right\|^2-\left\|b+c\right\|^2-\left\|a\right\|^2 \\
& \quad+\left\|a-c\right\|^2+\left\|b\right\|^2+\left\|b-c\right\|^2+\left\|a\right\|^2 \\
= &\, \left\|a+b+c\right\|^2-\left\|a+b-c\right\|^2 \\
& \quad-\frac{1}{2}\left(\left\|a+b+c\right\|^2+\left\|a-b+c\right\|^2\right)-\frac{1}{2}\left(\left\|a+b+c\right\|^2+\left\|b-a+c\right\|^2\right) \\
& \quad+\frac{1}{2}\left(\left\|a+b-c\right\|^2+\left\|a-b+c\right\|^2\right)+\frac{1}{2}\left(\left\|a+b-c\right\|^2+\left\|b-a+c\right\|^2\right) \\
= &\, 0.
\end{align*}
\end{proof}

\begin{proposition}
    Let $\mathcal{K}: X \to X$ be an $\eta$-strongly monotone operator, let $p \in X$, and let $\lambda \in \left(0,2\,\tfrac{1+\eta}{1+2\eta}\right]$.
    Define
    \begin{align*}
        x &= J_{\mathcal{K}}(p), \\
        p^{+} &= p + \lambda (x-p).
    \end{align*}
    Suppose there exists a point $p^* \in X$ such that
    $p^* = J_{\mathcal{K}}(p^*)$, then
    $$\norm{p^{+} - p^*} \leq \frac{1+(1-\lambda)\eta}{1+\eta} \norm{p-p^*}.$$
\end{proposition}

\begin{proof}
    By the definition of the resolvent, we have
    $$0 \in \mathcal{K}p^* \quad\text{and}\quad p-x \in \mathcal{K}x.$$
    If $\mathcal{K}$ is $\eta$-strongly monotone, it follows that
    $$\scp{p-x}{x-p^*} \geq \eta \norm{x-p^*}^2.$$
    This leads to
    \begin{align*}
        1+(1-\lambda)\eta \geq 1+\left(1-\frac{2(1+\eta)}{1+2\eta}\right)\eta = \frac{1+\eta}{1+2\eta} > 0,
    \end{align*}
    and thus
    \begin{align*}
        &\,\norm{p^+ - p^*}^2 - \frac{\left(1+(1-\lambda)\eta\right)^2}{(1+\eta)^2} \norm{p-p^*}^2 \\
       \leq &\, \norm{p-p^*+\lambda (x-p)}^2 - \frac{\left(1+(1-\lambda)\eta\right)^2}{(1+\eta)^2} \norm{p-p^*}^2 \\ 
       &\quad+ \frac{2 \lambda (1+(1-\lambda)\eta)}{1+\eta} \left[ \scp{p-x}{x-p^*} - \eta \norm{x-p^*}^2\right] \\
       \leq &\, \norm{p-p^*+\lambda (x-p)}^2 - \frac{\left(1+(1-\lambda)\eta\right)^2}{(1+\eta)^2} \norm{p-p^*}^2 \\ 
       &\quad+ \frac{2 \lambda (1+(1-\lambda)\eta)}{1+\eta} \left[ \scp{p-x}{x-p+p-p^*} - \eta \norm{x-p+p-p^*}^2\right].
    \end{align*}
    Rearranging yields
    \begin{align*}
        &\;\norm{p^+ - p^*}^2 - \frac{\left(1+(1-\lambda)\eta\right)^2}{(1+\eta)^2} \norm{p-p^*}^2 \\
        =&\; \left(1-\frac{(1+(1-\lambda)\eta)^2}{(1+\eta)^2}\right) \norm{p-p^*}^2 + 2 \lambda \scp{p-p^*}{x-p} \\ 
        &\quad+ \lambda^2 \norm{x-p}^2 - \frac{2 \lambda (1+(1-\lambda)\eta)}{1+\eta} \norm{x-p}^2 - \frac{2 \lambda (1+(1-\lambda)\eta)}{1+\eta} \scp{p-p^*}{x-p} \\
        &\quad- \frac{2\lambda \eta (1+(1-\lambda)\eta)}{1+\eta} \norm{x-p}^2 - \frac{2\lambda \eta (1+(1-\lambda)\eta)}{1+\eta} \norm{p-p^*}^2 \\
        &\quad- \frac{4 \lambda \eta (1+(1-\lambda)\eta)}{1+\eta} \scp{p-p^*}{x-p} \\
        =&\;-\frac{\lambda \eta^2 (2-\lambda+2(1-\lambda)\eta)}{(1+\eta)^2} \norm{p-p^*}^2 - \frac{2 \lambda (2-\lambda+2(1-\lambda)\eta)}{1+\eta} \scp{p-p^*}{x-p} \\
        &\quad+ \frac{\lambda^2-2\lambda+( 3\lambda - 4) \lambda \eta-2 \lambda (1-\lambda) \eta^2}{1+\eta} \norm{x-p}^2 \\
        =&\;-\frac{\lambda \eta^2 (2-\lambda+2(1-\lambda)\eta)}{(1+\eta)^2} \norm{p-p^*}^2 - \frac{2 \lambda \eta (2-\lambda+2(1-\lambda)\eta)}{1+\eta} \scp{p-p^*}{x-p} \\
        &\quad- \lambda (2-\lambda+2(1-\lambda) \eta) \norm{x-p}^2 \\
        =&\; - \lambda (2-\lambda+2(1-\lambda)\eta) \norm{x-p+\frac{\eta}{1+\eta}(p-p^*)}^2 \\
        \leq&\; 0,
    \end{align*}
    where in the last step, we used the equivalence
    \begin{align*}
        2-\lambda+2(1-\lambda)\eta \geq 0 \iff \lambda \leq \tfrac{2+2\eta}{1+2\eta}.
    \end{align*}
    Therefore,
    \begin{align*}
        \norm{p^+ - p^*} \leq \frac{1+(1-\lambda)\eta}{1+\eta} \norm{p-p^*}.
    \end{align*}
\end{proof}

\begin{corollary}\label{cor:contraction-linear-convergence}
    Let $\mathcal{K}: X \to X$ be an $\eta$-strongly monotone operator, let $p \in X$, and let $\lambda \in \left(0,2\,\tfrac{1+\eta}{1+2\eta}\right]$. 

    Furthermore, let the sequence $\left((x^k,p^k)\right)_{k \in \NN}$ be generated from an initial value $p^0 \in X$ according to the iteration:
    \begin{align*}
        x^{k} &= J_{\mathcal{K}}(p^k), \\
        p^{k+1} &= p^{k} + \lambda (x^{k}-p^{k}).
    \end{align*}
    
    Suppose there exist points $x^*, p^* \in X$ such that
    $x^* = p^* = J_{\mathcal{K}}(p^*)$. Then, the sequence satisfies
    \begin{align*}
        \norm{p^{k} - p^*} \leq \left(\frac{1+(1-\lambda)\eta}{1+\eta}\right)^k \norm{p^0 - p^*}
    \end{align*}
    and converges linearly.
\end{corollary}

Thus, we have established the foundation for proving linear convergence.

\begin{theorem}\label{thm:pddrmm-linear-convergence}
    Let the conditions from Lemma~\ref{lm:conds-alpha-monotony} be satisfied and let
    $$c = \min\left\{\tfrac{\mu_G - \Tilde{\mu}_G}{\mu_G \Tilde{\mu}_G},\tfrac{\mu_{F^*} - \Tilde{\mu}_{F^*}}{\mu_{F^*} \Tilde{\mu}_{F^*}}\right\}.$$
    Then, the sequence $\left(u^k\right)_{k\in \NN}$ with $u^k = (x^k, y^k)^T,$ $k \in \NN,$ of the primal-dual Douglas-Rachford method with mismatched adjoint~\eqref{alg-part:pddr_mismatch} converges linearly to the fixed point $\hat{u} \in \left(\cA + \cB\right)^{-1}(0)$ for any \begin{equation*}        
    \tau \in \left(0,\tfrac{1}{\norm{A-V}}\right) \cap \left(0, c \right)\,\text{ and }\,\theta \in \left(0, \tfrac{2}{1+\alpha}\right)
    \end{equation*}
    with $\alpha \geq  \tfrac{\tau}{c-\tau}$ and the following holds:
    \begin{equation}
        \norm{u^N-\hat{u}} \in \mathcal{O}\left(\left( \frac{1+(1-\theta (1+\alpha))\eta}{1+\eta}\right)^N\right)
    \end{equation}
    with
    \begin{equation}\label{def:eta}
        \eta = \frac{4\gamma}{27} \min\left\{ \tfrac{ \upsilon}{(1+2\alpha)^2}, \tfrac{ \sigma}{4 \gamma^2 \norm{\cB + \Sigma_{\cB}}^2 + \left(\alpha + \gamma \max \{\Tilde{\mu}_G, \Tilde{\mu}_{F^*}\}\right)^2} \right\},
    \end{equation}
    where $\gamma = (1+\alpha)\tau$, $\sigma = \sigma_{\min}\left(\cB + \Sigma_{\cB}\right)$, and $\upsilon = \min\{\gamma_G - \mu_G, \gamma_{F^*} - \mu_{F^*}\}$.
\end{theorem}
\begin{proof}
The choice of $\alpha$ and the upper bounds on $\tau$ and $\theta$ are due to the same arguments as in \ref{cor:weak_convergence}.
Using Theorem~\ref{thm:equivalence_pddr_precond}, we show that the iterates of iteration~\ref{eq:degeneratedpreconditionedproximalpoint} with \(\lambda = \theta (1+\alpha)\) converge linearly.

To this end, we aim to show that \(M^{-1} \mathbb{A}_{\alpha}\) is \(M\)-strongly monotone. Therefore we define \(\mathbf{u} = (u,r,s)^T\) and \(\mathbf{u}^{\prime} = (u^{\prime},r^{\prime},s^{\prime})^T\). We already know that $(\cA - \Sigma_{\cA})$ defined by the equations \eqref{eq:pddrmm-splitting} and \eqref{eq:pddr-sigmas} is $\upsilon$-strongly monotone with $\upsilon = \min\{\gamma_G - \mu_G, \gamma_{F^*} - \mu_{F^*}\}$. Also, by Lemma~\ref{lm:conds-alpha-monotony}, the operator $\left(\cB + \Sigma_{\cB}\right)$ is monotone. Using $\scp{\Tilde{\mathbb{A}}_{\alpha} \textbf{u}- \Tilde{\mathbb{A}}_{\alpha} \textbf{u}^{\prime}}{\textbf{u}-\textbf{u}^{\prime}} \geq 0$ from Lemma~\ref{lm:conds-alpha-monotony}, the $\upsilon$-strong monotony of $(\cA - \Sigma_{\cA})$ and the definition of $\sigma = \sigma_{\min}\left(\cB + \Sigma_{\cB}\right)$, we obtain:
\begin{align*}
        \scp{\mathbb{A}_{\alpha} \textbf{u}- \mathbb{A}_{\alpha} \textbf{u}^{\prime}}{\textbf{u}-\textbf{u}^{\prime}} \geq&\; \gamma \scp{\left(\cA - \Sigma_{\cA}\right) u - \left(\cA - \Sigma_{\cA}\right) u^{\prime}}{u - u^{\prime}} \\
        &\;+ \gamma \scp{\left(\cB + \Sigma_{\cB}\right)(s-s^{\prime})}{s-s^{\prime}} + \scp{\Tilde{\mathbb{A}}_{\alpha} \textbf{u}- \Tilde{\mathbb{A}}_{\alpha} \textbf{u}^{\prime}}{\textbf{u}-\textbf{u}^{\prime}} \\
        \geq&\; \gamma \upsilon \norm{u-u^{\prime}}^2 + \gamma \sigma \norm{s-s^{\prime}}^2.
\end{align*}
If \(\Tilde{\mu}_G \Tilde{\mu}_{F^*} > \tfrac14 \norm{A-V}^2\), then \(\sigma > 0\).
Now, using the Lipschitz continuity of \(\cB_{\Sigma} \defeq (\cB + \Sigma_{\cB})\), whose Lipschitz constant is given by \(\ell \defeq \norm{\cB_{\Sigma}}\), we obtain:
\begin{align*}
        \scp{\mathbb{A}_{\alpha} \textbf{u}- \mathbb{A}_{\alpha} \textbf{u}^{\prime}}{\textbf{u}-\textbf{u}^{\prime}} \geq \gamma \upsilon \norm{u-u^{\prime}}^2 + \gamma (1-t) \sigma \norm{s-s^{\prime}}^2 + \tfrac{t \sigma}{\gamma \ell^2} \norm{\gamma \cB_{\Sigma} \left(s-s^{\prime}\right)}^2
    \end{align*}
for any \(t \in [0,1]\). At this point, we have not yet introduced the constraint on the image of \(M\). Ultimately, according to \cite[Prop. 2.5 + Prop. 2.18]{naldi}, to prove the \(M\)-monotonicity of \(M^{-1} \mathbb{A}_{\alpha}\), it suffices to consider only pairs \((\mathbf{u},v),(\mathbf{u}^{\prime},v^{\prime}) \in \gr \mathbb{A}_{\alpha}\) for which \(v,v^{\prime} \in \operatorname{Im}(M)\). This corresponds to:
\[
\operatorname{Im}(M) = \left\{ (v_1,v_2,v_3)^T \in (X \times Y)^3 \mid v_1 = v_2 = v_3\right\}.
\]
In particular, for \(v \in \mathbb{A}_{\alpha} \mathbf{u}\), we have \(v_2 = v_3\), and similarly, for \(v^{\prime} \in \mathbb{A}_{\alpha} \mathbf{u}^{\prime}\), we obtain \(v_2^{\prime} = v_3^{\prime}\).
Consequently, it follows that:
\begin{align*}
        - \gamma \cB_{\Sigma} \left(s-s^{\prime}\right) = - 2 \alpha \left(u - u^{\prime}\right) + \left( r - r^{\prime}\right) + \left(1 + \alpha - \gamma \Sigma_B\right)\left(s-s^{\prime}\right).
    \end{align*}
Substituting this into the upper inequality, we obtain  
\begin{align*}
    &\scp{\mathbb{A}_{\alpha} \textbf{u}- \mathbb{A}_{\alpha} \textbf{u}^{\prime}}{\textbf{u}-\textbf{u}^{\prime}} \\ 
    \geq&\; \gamma \upsilon \norm{u-u^{\prime}}^2 + \gamma (1-t) \sigma \norm{s-s^{\prime}}^2 \\ 
    &\;+ \tfrac{t \sigma}{\gamma \ell^2} \norm{- 2 \alpha \left(u - u^{\prime}\right) + \left( r - r^{\prime}\right) + \left((1 + \alpha)I - \gamma \Sigma_B\right)\left(s-s^{\prime}\right)}^2
\end{align*}
for any \(t \in [0,1]\). According to Lemma~\ref{lemma:parallelogramm}, it is 
\begin{align*}
    \norm{a+b+c}^2 \leq \norm{a+b-c}^2 + \norm{a+c}^2 + \norm{b+c}^2 \;\text{ for }a,b,c \in X\times Y.
\end{align*}
Choosing \(a,b,c\) and \(\Lambda_1,\dots,\Lambda_3 \in \mathcal{L}(X \times Y, X \times Y)\) such that  
\begin{align*}
    a+c &= \Lambda_1 (u-u^{\prime}), \\
    b+c &= \Lambda_2 (s-s^{\prime}), \\
    a+b-c &= \Lambda_3 \left(-2\alpha (u-u^{\prime}) + (r-r^{\prime}) + \left((1-\alpha)I - \gamma \Sigma_{\cB}\right) (s-s^{\prime})\right)&\text{and}\\
    a+b+c &= (u-u^{\prime}) + (r-r^{\prime}) + (s-s^{\prime})
\end{align*}
we obtain  
\begin{align*}
    c &= \frac{1}{3}\left[(a+c)+(b+c)-(a+b-c)\right] \\
    &= \frac{1}{3}\left[\left(\Lambda_1 + 2 \alpha \Lambda_3\right) (u-u^{\prime}) - \Lambda_3 (r-r^{\prime}) + \left(\Lambda_2 - (1-\alpha) \Lambda_3 + \gamma \Lambda_3 \Sigma_{\cB} \right)(s-s^{\prime})  \right]
\end{align*}
and with  
\begin{align*}
    a+b+c &= (a+c) + (b+c) - c \\
    &= \left(\frac23 \Lambda_1 - \frac23 \alpha \Lambda_3\right) (u-u^{\prime}) \\
    &\quad+ \left(\frac23 \Lambda_2 + \frac{1-\alpha}{3} \Lambda_3 - \frac{\gamma}{3} \Lambda_3 \Sigma_{\cB}\right)(s-s^{\prime}) + \frac13 \Lambda_3 (r-r^{\prime})
\end{align*}
the admissible choice of operators \(\Lambda_1,\dots,\Lambda_3\) is given by  
\begin{align*}
    I &= \frac13 \Lambda_3 \iff \Lambda_3 = 3 I, \\
    I &= \frac23 \Lambda_1 - 2 \alpha \iff \Lambda_1 = \frac32 (1+2\alpha) I, \\
    I &= \frac23 \Lambda_2 + (1-\alpha) I - \gamma \Sigma_{\cB} \iff \Lambda_2 = \frac32 \left(\alpha I +\gamma \Sigma_{\cB}\right).
\end{align*}
Thus, using the definition~\eqref{eq:pddr-sigmas} of \(\Sigma_{\cB}\) and the notation \(s = (s_1,s_2)^T\) with \(s_1 \in X\) and \(s_2 \in Y\), we obtain the estimate  
    \begin{align*}
        &\scp{\mathbb{A}_{\alpha} \textbf{u}- \mathbb{A}_{\alpha} \textbf{u}^{\prime}}{\textbf{u}-\textbf{u}^{\prime}} \\
        =&\; \tfrac{4\gamma \upsilon}{9(1+2\alpha)^2} \norm{\tfrac32 (1+2\alpha)\left(u-u^{\prime}\right)}^2 \\
        &\;+ \tfrac{4 \gamma (1-t) \sigma}{9 \left(\alpha + \gamma \Tilde{\mu}_G \right)^2} \norm{\tfrac{3}{2}\left(\alpha + \gamma \Tilde{\mu}_G \right)\left(s_1-s_1^{\prime}\right)}^2 \\
        &\;+ \tfrac{4 \gamma (1-t) \sigma}{9 \left(\alpha + \gamma \Tilde{\mu}_{F^*} \right)^2} \norm{\tfrac{3}{2}\left(\alpha + \gamma \Tilde{\mu}_{F^*} \right)\left(s_2-s_2^{\prime}\right)}^2 \\
        &\;+ \tfrac{t \sigma}{9 \gamma \ell^2} \norm{3\left(- 2 \alpha \left(u - u^{\prime}\right) + \left( r - r^{\prime}\right) + \left(1 + \alpha)I- \gamma \Sigma_B\right)\left(s-s^{\prime}\right)\right)}^2 \\
        \geq&\; 3\eta(t) \left[ \norm{\tfrac32 (1+2\alpha)\left(u-u^{\prime}\right)}^2  \right. \\
        &\;+ \norm{\tfrac{3}{2}\left(\alpha + \gamma \Tilde{\mu}_G \right)\left(s_1-s_1^{\prime}\right)}^2 + \norm{\tfrac{3}{2}\left(\alpha + \gamma \Tilde{\mu}_{F^*} \right)\left(s_2-s_2^{\prime}\right)}^2 \\
        &\;+\left.\norm{3\left(- 2 \alpha \left(u - u^{\prime}\right) + \left( r - r^{\prime}\right) + \left(1 + \alpha)I- \gamma \Sigma_B\right)\left(s-s^{\prime}\right)\right)}^2 \right] \\
        =&\; 3\eta(t) \left[ \norm{\tfrac32 (1+2\alpha)\left(u-u^{\prime}\right)}^2  + \norm{\tfrac{3}{2}\left(\alpha I + \gamma \Sigma_{\cB} \right)\left(s-s^{\prime}\right)}^2 \right. \\
        &\;+\left.\norm{3\left(- 2 \alpha \left(u - u^{\prime}\right) + \left( r - r^{\prime}\right) + \left(1 + \alpha)I- \gamma \Sigma_B\right)\left(s-s^{\prime}\right)\right)}^2 \right] \\
        \geq&\; 3 \eta(t) \norm{(u+r+s)-(u^{\prime}+r^{\prime}+s^{\prime})}^2 \\
        =&\; \eta(t) \norm{\textbf{u} - \textbf{u}^{\prime}}^2_M
    \end{align*}
    with
    \begin{equation*}
        \eta(t) = \frac{1}{27} \min\left\{ \tfrac{4 \gamma \upsilon}{(1+2\alpha)^2}, \tfrac{4 \gamma (1-t) \sigma}{\left(\alpha + \gamma \max \{\Tilde{\mu}_G, \Tilde{\mu}_{F^*}\}\right)^2}, \tfrac{t \sigma}{\gamma \ell^2} \right\}.
    \end{equation*}
Now, the second term in the minimum is decreasing in \(t\) and equals zero for \(t=1\), while the last term vanishes at \(t=0\) and increases with \(t\). Thus, we maximize the minimum by computing \(\hat{t}\), where both terms are equal. This leads to  
\begin{equation*}
    \hat{t} = \frac{4 \gamma^2 \ell^2}{4 \gamma^2 \ell^2 + \left(\alpha + \gamma \max \{\Tilde{\mu}_G, \Tilde{\mu}_{F^*}\}\right)^2}
\end{equation*}
and  
\begin{equation*}
    \eta = \eta(\hat{t}) = \frac{4}{27} \min\left\{ \tfrac{\gamma \upsilon}{(1+2\alpha)^2}, \tfrac{ \gamma \sigma}{4 \gamma^2 \ell^2 + \left(\alpha + \gamma \max \{\Tilde{\mu}_G, \Tilde{\mu}_{F^*}\}\right)^2} \right\}.
\end{equation*}
Accordingly, by \cite[Prop. 2.18]{naldi}, \(M^{-1} \mathbb{A}_{\alpha}\) is \(M\)-\(\eta\)-strongly monotone. By \cite[Th. 2.19]{naldi}, for \(M = CC^*\), the operator \((C^* \mathbb{A}_{\alpha} C)^{-1}\) is \(\eta\)-strongly monotone. From \cite[Th. 2.13]{naldi}, it follows that \(C^*(M+\mathbb{A}_{\alpha})^{-1}C\) corresponds to the resolvent \(J_{\left(C^* \mathbb{A}_{\alpha}^{-1} C\right)^{-1}}\).  
Thus, for the iteration~\eqref{eq:degeneratedpreconditionedproximalpoint}, by Corollary~\ref{cor:contraction-linear-convergence}, we obtain  
\begin{align*}
    \norm{w^{k+1} - \hat w} \leq \left(\frac{1+(1-\lambda)\eta}{1+\eta}\right)^k \norm{w^0 - \hat w}
\end{align*}
and therefore, by Theorem~\ref{thm:equivalence_pddr_precond} and due to the nonexpansiveness of the resolvent \(J_{\left(C^* \mathbb{A}_{\alpha}^{-1} C\right)^{-1}}\), the sequence of variables \(\left(u^k\right)_{k\in \NN}\) with \(u^k = (x^k, y^k)^T, k \in \NN\), converges linearly to the fixed point \(\hat{u}\) with  
\[
\norm{u^N-\hat{u}} \in \mathcal{O}\left( \left( \frac{1+(1-\theta (1+\alpha))\eta}{1+\eta} \right)^N\right).
\]
\end{proof}

At this point, we have achieved the goal of proving the linear convergence of the primal-dual Douglas-Rachford method with mismatched adjoint. Furthermore, we have provided a convergence rate that can be easily computed based on the given parameters $\tau$ and $\theta$ of the algorithm and the existing problem formulation.

\subsection{Choice of Parameters}
We now discuss how to choose suitable parameters $\tau$ and $\theta$ for fastest convergence while keeping the computation straightforward. However, we will avoid delving too deeply into technical analysis in some parts to maintain clarity in the step size selection.

Let $G$ be a $\gamma_G$-strongly convex function and $F^*$ a $\gamma_{F^*}$-strongly convex function satisfying
$$\gamma_G \gamma_{F^*} > \frac14 \norm{A-V}^2.$$
We begin by establishing feasible choices for $\mu_G$, $\mu_{F^*}$, $\Tilde{\mu}_G$, and $\Tilde{\mu}_{F^*}$ such that the conditions $0 < \Tilde{\mu}_G < \mu_G < \gamma_G$, $0 < \Tilde{\mu}_{F^*} < \mu_{F^*} < \gamma_{F^*}$, and
$\Tilde{\mu}_G \Tilde{\mu}_{F^*} > \tfrac14 \norm{A-V}^2$
are satisfied. 

Thus, we set
\begin{align*}
    &\gamma_G > \Tilde{\mu}_G > \tfrac{\norm{A-V}}{2} \sqrt{\tfrac{\gamma_G}{\gamma_{F^*}}},
    &\mu_G = \tfrac{1}{2}\left(\gamma_G + \Tilde{\mu}_G \right), \\ &\gamma_{F^*} > \Tilde{\mu}_{F^*} > \tfrac{\norm{A-V}}{2} \sqrt{\tfrac{\gamma_{F^*}}{\gamma_{G}}}, & \mu_{F^*} = \tfrac{1}{2}\left(\gamma_{F^*} + \Tilde{\mu}_{F^*} \right)
\end{align*}
and thereby also fix the quantities $\sigma = \sigma_{\min}\left(\cB + \Sigma_{\cB}\right)$, $\norm{\cB + \Sigma_{\cB}}$, and
\begin{align*}
    \upsilon &= \min\{\gamma_G - \mu_G, \gamma_{F^*} - \mu_{F^*}\} = \frac12 \min\left\{\gamma_G - \mu_G, \gamma_{F^*} - \mu_{F^*}\right\} > 0.
\end{align*}

To obtain the best possible convergence rate, we fix the extrapolation parameter $\theta \in (0,1)$ and choose $\tau$ optimally as a function of $\theta$. 

Given $\theta \in (0,1)$, we set $\alpha = \tfrac{1}{\theta} - 1$ to optimize the convergence rate. This results in
$\theta = \tfrac{1}{1+\alpha}$ and a convergence rate of $\tfrac{1}{1+\eta}$. 

Since the reciprocal of $\theta$ appears frequently in the subsequent derivations, we define $\delta \defeq \theta^{-1} > 1$ and replace occurrences of $\tfrac{1}{\theta}$ with it.

The chosen value $\alpha > 0$ imposes upper bounds on the step size $\tau$ through the conditions in~\eqref{eq:conds-alpha}, given that $\gamma = (1+\alpha) \tau$ from \eqref{eq:def-tau-gamma}. Specifically, we have
\begin{equation*}
    \tau \leq \tfrac{\alpha}{1+\alpha} \min\left\{\tfrac{\mu_G - \Tilde{\mu}_G}{\mu_G \Tilde{\mu}_G}, \tfrac{\mu_{F^*} - \Tilde{\mu}_{F^*}}{\mu_{F^*} \Tilde{\mu}_{F^*}}\right\}.
\end{equation*}
Additionally, since $\tau < \norm{A-V}^{-1}$ must be satisfied according to \eqref{eq:def-tau-gamma}, we define an upper bound $\tau_S$ on $\tau$ as
$$\tau_S \defeq \tfrac{\delta - 1}{\delta}\min\left\{\tfrac{\mu_G - \Tilde{\mu}_G}{\mu_G \Tilde{\mu}_G}, \tfrac{\mu_{F^*} - \Tilde{\mu}_{F^*}}{\mu_{F^*} \Tilde{\mu}_{F^*}}, \tfrac{0,99 \delta}{(\delta-1) \norm{A-V}}\right\}.$$

Given the chosen parameters, we aim to minimize the convergence rate $\tfrac{1}{1+\eta}$, which requires maximizing the term $\eta$ as defined in \eqref{def:eta}. 

Since $\gamma = \delta \tau$, we maximize the term
\begin{equation*}
    \eta = \frac{4}{27} \min\bigg\{ \underbrace{\tfrac{\tau \delta \upsilon}{(2 \delta - 1)^2}}_{f_1(\tau)}, \underbrace{\tfrac{ \tau \delta \sigma}{4 \tau^2 \delta^2 \norm{\cB+\Sigma_{\cB}}^2 + (\delta-1+\tau \delta \max\{\tilde{\mu}_G,\tilde{\mu}_{F^*}\})^2}}_{f_2(\tau)}\bigg\}
\end{equation*}
over the remaining free parameter $\tau$. 

Clearly, $f_1$ grows linearly in $\tau$, while $f_2$ is a rational function of $\tau$. Both functions vanish at $\tau = 0$, as illustrated in Figure~\ref{fig:f1+f2}, which presents different function behaviors. 

Moreover, $f_2(\tau)$ increases for $0 \leq \tau \leq \zeta$ and decreases for $\tau > \zeta$, where
\begin{equation*}
    \zeta \defeq \sqrt{\tfrac{(\delta-1)^2}{\delta^2 (4 \norm{\cB+\Sigma_{\cB}}^2 + \max\{\tilde{\mu}_G^2,\tilde{\mu}_{F^*}^2\})}}.
\end{equation*}

\begin{figure}[h]
    \centering 
    \begin{tikzpicture}[xscale=2]
        \draw[->] (-0.2,0) -- (2,0) node[below]{$\tau$};
        \draw[->] (0,-0.2) -- (0,1.3);
        \begin{scope}
            \clip (-0.5,-0.5) rectangle (3,1);
            \draw[domain=0:2,thick,black!40] plot (\x,{0.5*\x});
            \draw[domain=0:2,thick,black!80] plot (\x,{0.5*\x/(\x^2*1+(1-\x*0.1)^2))});
        \end{scope}
        \node[black!40] at (1,1)[right]{\small $f_1(\tau)$};
        \node[black!80] at (1.7,0.2)[above]{\small $f_2(\tau)$};
    \end{tikzpicture}
    \begin{tikzpicture}[xscale=2]
        \draw[->] (-0.2,0) -- (2,0) node[below]{$\tau$};
        \draw[->] (0,-0.2) -- (0,1.3);
        \begin{scope}
            \clip (-0.5,-0.5) rectangle (3,1);
            \draw[domain=0:2,thick,black!40] plot (\x,{0.5*\x});
            \draw[domain=0:2,thick,black!80] plot (\x,{0.5*\x/(\x^2*0.8+(0.5-\x*0.1)^2))});
        \end{scope}
        \node[black!40] at (1,1)[right]{\small $f_1(\tau)$};
        \node[black!80] at (1.8,0.25)[above]{\small $f_2(\tau)$};
    \end{tikzpicture}
    \caption{Different possible behaviors of the component functions $f_1$ and $f_2$ in the definition of $\eta$.}
    \label{fig:f1+f2}
\end{figure}
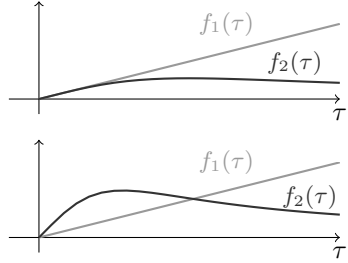

Accordingly, either $f_1(\tau) > f_2(\tau)$ for all $\tau > 0$, or $f_1$ and $f_2$ have an additional intersection point. If the origin is the only intersection, then 
$$\tau = \min\{\zeta, \tau_S\},$$
since we must not exceed the previously established upper bound $\tau_S$ for $\tau$.

Otherwise, the maximum value of $\eta$ occurs at an intersection point where $f_1(\tau) = f_2(\tau)$ or is determined by the value at the upper bound or the maximizer of $f_2$. The intersection points $\tau_{-}$ and $\tau_{+}$ are given by

\begin{align*}
    \tau_{\pm} =&\; \tfrac{ (1-\delta) \max\{\tilde{\mu}_G,\tilde{\mu}_{F^*}\}}{\delta \left(4 \norm{\cB+\Sigma_{\cB}}^2 + \max\{\tilde{\mu}_G^2,\tilde{\mu}_{F^*}^2\}\right)} \\
    &\pm \tfrac{\sqrt{ (\delta-1)^2 \max\{\tilde{\mu}^2_G,\tilde{\mu}^2_{F^*}\} - \left((\delta-1)^2-\tfrac{\sigma}{\upsilon}(2\delta-1)^2\right)\, \left(4 \norm{\cB+\Sigma_{\cB}}^2 + \max\{\tilde{\mu}_G^2,\tilde{\mu}_{F^*}^2\}\right)}}{\delta \left(4 \norm{\cB+\Sigma_{\cB}}^2 + \max\{\tilde{\mu}_G^2,\tilde{\mu}_{F^*}^2\}\right)}.
\end{align*}

Thus, if there is exactly one intersection point $\tau_+$ on the positive real half-axis, we set $\tau = \min\{\tau_+, \tau_S\}$, yielding $\eta = \tfrac{4}{27} f_1(\tau)$. However, if there are two intersection points on the positive real half-axis, we choose $\tau = \min\{\tau_+, \tau_S\}$ if $\zeta < \tau_+$, and otherwise, $\tau = \min\{\zeta, \tau_S\}$. The convergence rate then follows from the previous observations.

Consequently, we can summarize these results in a straightforward theorem for easy application.

\begin{theorem}\label{thm:pddrmm-linear-convergence-easy}
    Let $G$ be a $\gamma_G$-strongly convex function and $F^*$ a $\gamma_{F^*}$-strongly convex function such that
    $$\gamma_G \gamma_{F^*} >  \frac14 \norm{A-V}^2$$
    holds. Choose 
    \begin{align*}
    &\gamma_G > \Tilde{\mu}_G > \tfrac{\norm{A-V}}{2} \sqrt{\tfrac{\gamma_G}{\gamma_{F^*}}},
    &\mu_G = \tfrac{1}{2}\left(\gamma_G + \Tilde{\mu}_G \right), \\ &\gamma_{F^*} > \Tilde{\mu}_{F^*} > \tfrac{\norm{A-V}}{2} \sqrt{\tfrac{\gamma_{F^*}}{\gamma_{G}}}, & \mu_{F^*} = \tfrac{1}{2}\left(\gamma_{F^*} + \Tilde{\mu}_{F^*} \right)
    \end{align*}
    and define
    $$\cB_{\Sigma} = \begin{pmatrix} \Tilde{\mu}_G I & V^* \\ -A & \Tilde{\mu}_{F^*} I \end{pmatrix}$$
    as well as $\sigma = \sigma_{\min}(\cB_{\Sigma})$. \newline\noindent
    Choose the extrapolation parameter $\theta \in (0,1)$ and define $\delta = \theta^{-1}$. Furthermore, set
    \begin{align*}
        \zeta &= \tfrac{\delta-1}{\delta \sqrt{4\norm{\cB_{\Sigma}}^2 + \max\{\tilde{\mu}_G^2,\tilde{\mu}_{F^*}^2\}}}, \\
        \tau_S &= \tfrac{\delta - 1}{\delta}\min\left\{\tfrac{\mu_G - \Tilde{\mu}_G}{\mu_G \Tilde{\mu}_G}, \tfrac{\mu_{F^*} - \Tilde{\mu}_{F^*}}{\mu_{F^*} \Tilde{\mu}_{F^*}}, \tfrac{0.99 \delta}{(\delta-1) \norm{A-V}}\right\}, \\
        \upsilon &= \frac12 \min\left\{\gamma_G - \mu_G, \gamma_{F^*} - \mu_{F^*}\right\}
    \end{align*}
    and define $\tau_{-}$ and $\tau_{+}$ by
    \begin{align*}
    \tau_{\pm} =&\; \tfrac{ (1-\delta) \max\{\tilde{\mu}_G,\tilde{\mu}_{F^*}\}}{\delta \left(4\norm{\cB_{\Sigma}}^2 + \max\{\tilde{\mu}_G^2,\tilde{\mu}_{F^*}^2\}\right)} \\
    &\pm \tfrac{\sqrt{ (\delta-1)^2 \max\{\tilde{\mu}^2_G,\tilde{\mu}^2_{F^*}\} - \left((\delta-1)^2-\tfrac{\sigma}{\upsilon}(2\delta-1)^2\right)\, \left(4\norm{\cB_{\Sigma}}^2 + \max\{\tilde{\mu}_G^2,\tilde{\mu}_{F^*}^2\}\right)}}{\delta \left(4\norm{\cB_{\Sigma}}^2 + \max\{\tilde{\mu}_G^2,\tilde{\mu}_{F^*}^2\}\right)}.
    \end{align*}
    Now set
    \begin{equation*}
        \tau = \min\{\tau_S, \tilde{\tau}\} \quad\text{with}\quad \tilde{\tau} = \left\{ \begin{array}{ll}
            \zeta, & (\tau_{-},\tau_{+} \notin \RR) \lor (\tau_{-} > 0 \land \zeta > \tau_{+} > 0),\\
            \tau_+, & (\tau_{-} < 0) \lor (\tau_{-} > 0 \land \tau_{+} \geq \zeta > 0).
        \end{array} \right.
    \end{equation*}
    Then the sequence $\left(u^k\right)_{k\in \NN}$ with $u^k = (x^k, y^k)^T, k \in \NN,$ of the primal-dual Douglas-Rachford method with mismatched adjoint~\eqref{alg-part:pddr_mismatch} with step size $\tau$ and extrapolation parameter $\theta$ converges linearly to the unique fixed point $\hat{u} \in \left(\cA + \cB\right)^{-1}(0)$, satisfying
    \begin{equation*}
        \norm{u^N-\hat{u}} \in \mathcal{O}\left(\tfrac{1}{\left( 1+\eta \right)^N}\right)
    \end{equation*}
    with
    \begin{equation*}
        \eta = \frac{4\tau \delta}{27} \min\left\{\tfrac{\upsilon}{(2 \delta - 1)^2}, \tfrac{\sigma}{4 \tau^2 \delta^2 \norm{}^2 + (\delta-1+\tau \delta \max\{\tilde{\mu}_G,\tilde{\mu}_{F^*}\})^2} \right\}.
    \end{equation*}
\end{theorem}

This theorem thus provides us with a straightforward way to determine step sizes that ensure linear convergence. Of course, one could still optimize the convergence rate by selecting $\theta$ appropriately. However, to maintain clarity, we will refrain from doing so at this point. The optimization itself is relatively simple, albeit tedious, and involves little more than the Cardano formulas.

The question that still remains is whether we can achieve convergence without requiring the strong convexity of both functions. The theoretical foundation on which the proof is built does not yield a corresponding result. In this theory, the presence of monotone operators is crucial for the applicable theorems. A key component of the approach in the previous discussion was the analysis of the operator
$$\mathbb{A}_{\alpha} = \left[\begin{array}{cccccc}
\alpha I + \gamma \partial G & 0 & -I & 0 & -I & 0 \\ 
0 & \alpha I + \gamma \partial F^{*} & 0 & -I & 0 & -I \\
I & 0 & 0 & 0 & -I & 0 \\
0 & I & 0 & 0 & 0 & -I \\
(1-2\alpha) I & 0 & I & 0 & \alpha I & \gamma V^* \\
0 & (1-2\alpha) I & 0 & I & - \gamma A & \alpha I
\end{array}\right].$$
By carefully choosing the parameters, we ensured that this operator is monotone in the preceding theorems and lemmas. This allowed us to conclude convergence from Theorems~\cite[Th. 2.14]{naldi} and~\cite[Th. 2.19]{naldi}. However, without assuming strong convexity of both functions, $M^{-1} \mathbb{A}_{\alpha}$ is not $M$-monotone, as can be seen from the vector $v = (a,b,a,b,a,b)^T$ and 
\begin{align*}
    \scp{v}{\mathbb{A}_{\alpha}v} =  \alpha \scp{a}{\partial G(a)} - \gamma \scp{(A-V)a}{b} + \alpha \scp{b}{\partial F^*(b)}.
\end{align*}
Furthermore, Example~\ref{pddr:counter-example} illustrates that without the assumption of strong convexity for both functions, the primal-dual Douglas-Rachford method with mismatched adjoint may not converge to a fixed point of the iteration, even if one exists.

\section{Numerical Experiments}
\label{sec:numerical-experiments}
\subsection{Convex Quadratic Problems}
\label{sec:pddr-ex-quadratic}

As an example to visualize the proven results, we examine convex quadratic problems of the form

\begin{align*}
  \min_{x\in\RR^{n}}\tfrac{\alpha}{2}\norm{x}_{2}^{2} + \tfrac{1}{2\beta}\norm{Ax-z}_{2}^{2}
\end{align*}
with $\alpha,\beta>0$, $A\in\RR^{m\times n}$, and $z\in\RR^{m}$. Defining $G(x) = \tfrac\alpha2\norm{x}_{2}^{2}$ and $F^{*}(y) = \tfrac\beta2\norm{y}_{2}^{2} + \scp{y}{z}$, the corresponding proximal operators are given by
\begin{align*}
  \prox_{\tau G}(x) = \tfrac{x}{1+\tau\alpha} \quad \text{and} \quad \prox_{\tau F^{*}}(y) = \tfrac{y-\tau z}{1+\tau\beta}.
\end{align*}
Furthermore, $G$ is strongly convex with constant $\gamma_{G} = \alpha$ and $F^{*}$ is strongly convex with constant $\gamma_{F^{*}} = \beta$. Therefore, for $\alpha,\beta>0$, we can use Theorem~\ref{thm:pddrmm-linear-convergence-easy} to determine a valid step size $\tau$ for a given choice of $\theta$.

The resulting algorithm is:
\begin{align}\label{eq:pddrmm-quadratic}
    \begin{split}
    x^{k+1} &= \tfrac{p^k}{1+\tau\alpha}, \\
    y^{k+1} &= \tfrac{q^{k}-\tau z}{1+\tau \beta}, \\
    \begin{bmatrix}v^{k+1} \\ w^{k+1} \end{bmatrix} &= \begin{bmatrix} I & \tau V^* \\ - \tau A & I \end{bmatrix}^{-1} \begin{bmatrix} 2 x^{k+1} - p^k \\ 2 y^{k+1} - q^k \end{bmatrix}, \\
    p^{k+1} &= p^k + \theta \left( v^{k+1} - x^{k+1} \right), \\
    q^{k+1} &= q^k + \theta \left( w^{k+1} - y^{k+1} \right).
    \end{split}
\end{align}
For our numerical experiment, we again set $n=400$, $m=200$, a random matrix $A\in\RR^{m\times n}$, and a perturbation $V\in\RR^{m\times n}$ by adding a small random matrix to $A$, i.e.,
\begin{align*}
  V = A + E\ \text{ with }\ \norm{E}\leq \eta.
\end{align*}
We numerically verify the condition $\gamma_{G}\gamma_{F^{*}}>\tfrac14 \norm{A-V}^{2}$ and use Theorem~\ref{thm:pddrmm-linear-convergence-easy} to determine step sizes that guarantee convergence.

The unique fixed points of the iteration are:
\begin{equation}
  \label{eq:pdrmmm-quadratic-fp-mm}
  \begin{split}
  \hat x & = \Big(\alpha I + \tfrac{1}{\beta}V^{T}A\Big)^{-1}(\tfrac1\beta V^{T}z) = V^T(\alpha\beta I + AV^{T})^{-1}z,\\
  \hat y & = -(\beta I + \tfrac{1}{\alpha}AV^{T})^{-1}z = -\alpha(\alpha\beta I + AV^{T})^{-1}z,
  \end{split}
\end{equation}
while the true primal solution is:
\begin{align}
  \label{eq:pddrmm-quadratic-fp}
  x^{*} = \Big(\alpha I + \tfrac{1}{\beta}A^{T}A\Big)^{-1}(\tfrac1\beta A^{T}z)  = A^T(\alpha\beta I + AA^{T})^{-1}z.
\end{align}
For the numerical analysis, we set $\alpha = \gamma_G=0.15$ and $\beta = \gamma_{F^{*}}=1$ in Theorem~\ref{thm:pddrmm-linear-convergence-easy}. Additionally, we choose $\theta = 0.5$. This leads to the step size $\tau \approx 0.2$ using:
$$\tilde{\mu}_G = \frac12 \left(\gamma_G + \tfrac{\norm{A-V}}{2} \sqrt{\tfrac{\gamma_G}{\gamma_{F^*}}}\right)\quad\text{ and }\quad\tilde{\mu}_{F^*} = \frac12 \left(\gamma_{F^*} + \tfrac{\norm{A-V}}{2} \sqrt{\tfrac{\gamma_{F^*}}{\gamma_{G}}}\right).$$
Figure~\ref{fig:pddrmm-quadratic_linear_conv} presents the results for the mismatched adjoint method. As expected, we observe linear convergence toward the fixed point $\hat{x}$, and the iterates reach the predicted error bound relative to the true minimizer $x^{*}$, as given by Theorem~\ref{thm:error-estimate}.

\begin{figure}[H]
\centering
\includegraphics[width=6cm]{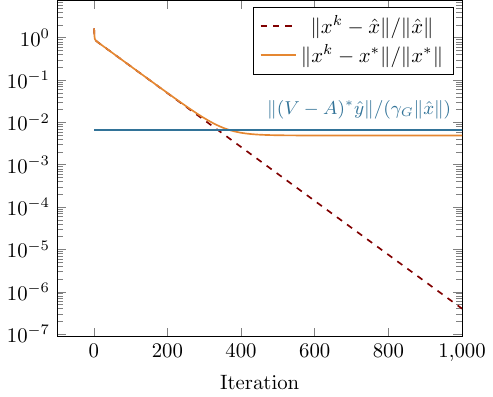} 
\caption{Convergence of the iteration~\eqref{eq:pddrmm-quadratic}. Here, $\hat x$ is the fixed point~\eqref{eq:pdrmmm-quadratic-fp-mm} of the iteration with the incorrect adjoint, and $x^{*}$ is the original primal solution~\eqref{eq:pddrmm-quadratic-fp}. The solid orange line represents the relative distance of the primal iterates $x^{k}$ from the fixed point of the iteration, while the dashed purple line represents the relative distance of the iterates $x^{k}$ from the original primal solution. As predicted, the latter distance remains below the bound given in Theorem~\ref{thm:error-estimate}.}
\label{fig:pddrmm-quadratic_linear_conv}
\end{figure}
Additionally, we can compare the Chambolle-Pock method with mismatched adjoint from~\cite{cp_mismatched} with the adapted primal-dual Douglas-Rachford method with mismatched adjoint. The iteration for the adapted method is given by:
\begin{align}\label{eq:adapted-pddrmm-quadratic}
    \begin{split}
    x^{k+1} &= \tfrac{p^k}{1+\tau (\alpha- \mu_G)}, \\
    y^{k+1} &= \tfrac{q^{k}-\tau z}{1+\tau (\beta-\mu_{F^*})}, \\
    \begin{bmatrix}v^{k+1} \\ w^{k+1} \end{bmatrix} &= \begin{bmatrix} (1+\tau \mu_G)I & \tau V^* \\ - \tau A & (1+\tau \mu_{F^*}) I \end{bmatrix}^{-1} \begin{bmatrix} 2 x^{k+1} - p^k \\ 2 y^{k+1} - q^k \end{bmatrix}, \\
    p^{k+1} &= p^k + \theta \left( v^{k+1} - x^{k+1} \right), \\
    q^{k+1} &= q^k + \theta \left( w^{k+1} - y^{k+1} \right).
    \end{split}
\end{align}
We set the parameters $\alpha, \beta, \theta$, and $\tau$ exactly as before. Furthermore, we set the parameters $\mu_G$ and $\mu_{F^*}$ of the adapted method equal to the previously used $\tilde{\mu}_G$ and $\tilde{\mu}_{F^*}$, respectively.

Figure~\ref{fig:adapted-pddrmm-quadratic_linear_conv} shows the results of the adapted method. As expected, it converges to the same limit, and the convergence rates of both methods do not differ significantly in this example.
\begin{figure}[H]
\centering
\includegraphics[width=6cm]{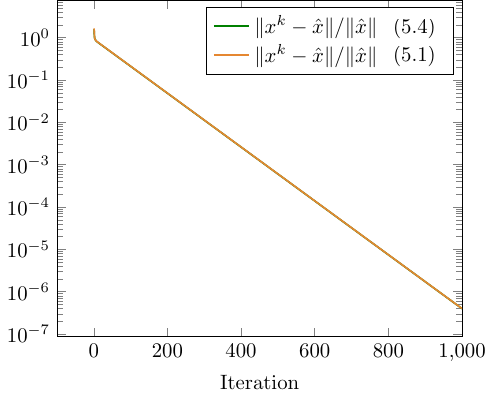} 
\caption{Comparison of the convergence of iterates in the adapted and non-adapted primal-dual Douglas-Rachford method with mismatched adjoint to the fixed point. The green line shows the relative distance of the iterates of the adapted method to the fixed point, while the orange line represents the relative distance of the iterates of the non-adapted primal-dual Douglas-Rachford method to the fixed point. In this example, the convergence speed is nearly identical.}
\label{fig:adapted-pddrmm-quadratic_linear_conv}
\end{figure}
The comparison with the Chambolle-Pock method with mismatched adjoint from \cite{cp_mismatched} is visualized in Figure~\ref{fig:cp-pddrmm-quadratic_linear_conv}. It is evident that the primal-dual Douglas-Rachford method with mismatched adjoint converges significantly faster. However, we do have to mention that this is only of practical relevance, if the structures of $A$ and $V$ allow a fast computation of the solution of the linear system.
\begin{figure}[H]
\centering
\includegraphics[width=6cm]{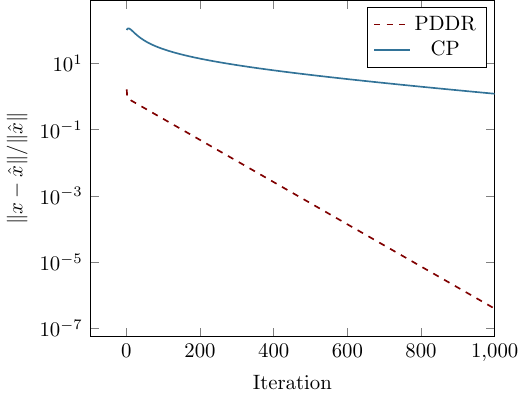} 
\caption{Comparison of the convergence of iterates in the primal-dual Douglas-Rachford method with mismatched adjoint versus the Chambolle-Pock method with mismatched adjoint. The blue line represents the relative distance between the fixed point and the iterates of the Chambolle-Pock method, while the red line represents the relative distance of the iterates of the primal-dual Douglas-Rachford method with mismatched adjoint.}
\label{fig:cp-pddrmm-quadratic_linear_conv}
\end{figure}

\subsection{Computerized tomography}
\label{sec:computerized-tomography}
To illustrate the impact on one of the standard examples of inverse problems, we consider the problem of computerized tomography (CT)~\cite{buzug2008computed}. In computerized tomography we aim to reconstruct a slice of an object from x-ray measurements taken in different directions. These measurements are stored as a sinogram and the map of the slice to the sinogram is modeled by a linear map: the Radon transform or forward projection. Its adjoint is called backprojection. There exist various inversion formulas, which express the inverse of the Radon transform explicitly, but since the Radon transform is compact (when modeled as a map between the right function spaces~\cite{natterer2001mathematics}), any inversion formula has to be unstable. A popular stable, approximate inversion method is the so called filtered backprojection (FBP)~\cite{buzug2008computed}. The method gives good approximate reconstruction when the number of projections is high and when the data is not too noisy. However, the quality of the reconstruction quickly decreases when the number of projections decreases. But as this lowers the x-ray doses, there are numerous efforts to increase reconstruction quality from only a few projections. A successful approach uses total variation (TV) regularization~\cite{sidky2012convex}. The respective minimization problem can then be solved with the primal-dual Douglas-Rachford method. However, there are many ways to implement the forward and the backward projection. In applications it sometimes happens that a pair of forward and backward projections are chosen that are not adjoint to each other, either because this importance of adjointness is not noted, or on purpose to enhance the speed of computation or to achieve a certain reconstruction quality, see also~\cite{zeng2000unmatched,Palenstijn2011PerformanceIF,Chouzenout2021pgm-adjoint,Lorenz2018TheRK}. 

In our experiment, we describe a discrete image with $m\times n$ pixels as $x\in\RR^{m\times n}$. Its discrete gradient $\nabla x = u\in\RR^{m\times n\times 2}$ is a tensor and for such a tensor we define the pixel-wise absolute value in the pixel $(i_{1},i_{2})$ as $\abs{u}_{i_{1},i_{2}}^2 = \sum_{k=1}^2 u_{i_{1},i_{2},k}^2$, cf.~\cite[p. 416]{Bredies2019MathematicalIP}. For images $x\in\RR^{m\times n}$ we denote by 
$$\norm{x}_{p} = \left( \sum_{i_{i},i_{2}}\abs{x}_{i_{1},i_{2}}^{p} \right)^{1/p}$$ the usual pixel-wise $p$-norm.
With $R$ we denote the discretized Radon transform taking a $m\times n$-pixel image to a sinogram of size $s\times t$. 
We aim to solve the problem
\[
\min _{x \in  \mathbb{R}^{m \times n}} \frac{\lambda_0}{2}\|R x-z\|_{2}^{2}+ \frac{\lambda_1}{2} \norm{\abs{\nabla x}}_{1}+ \frac{\lambda_2}{2} \norm{x}_2^2
\]
for a given sinogram $z$ and constants $\lambda_{0},\lambda_1,\lambda_{2}>0$.
This can be expressed as the saddle point problem
\[
\min_{x \in \mathbb{R}^{m \times n}} \max _{\substack{p \in \mathbb{R}^{m \times n \times 2}\\ q \in \mathbb{R}^{s \times t}}}- \langle x, \operatorname{div} p\rangle + \langle R x-z, q\rangle - \frac{1}{2 \lambda_0} \norm{q}^2 - I_{\norm{\cdot}_{\infty,2} \leq \lambda_1}(p) + \frac{\lambda_2}{2} \norm{x}^2.
\]
With $F^*(q,p) = 
\tfrac{1}{2 \lambda_0}\|q\|_{2}^{2} + \langle q, z \rangle + I_{\norm{\cdot}_{\infty,2} \leq \lambda_1}(p)$, $G(x) = \tfrac{\lambda_2}{2} \norm{x}_2^2$ and $A =
\left(\begin{array}{l}
R \\
\nabla
\end{array}\right)
$ the saddle point formulation is exactly of the form~(\ref{eq:minFA+G}).
The function $G$ is strongly convex, however, $F^{*}$ is not. Hence, we regularize further by adding $\epsilon\|p\|_{2}^{2}/2$ with $\epsilon>0$ to $F^{*}$ which amounts to a Huber-smoothing of the total variation term in the primal problem. Therefore the proximal mappings are given by
\begin{equation}\label{eq:prox_G}
    \prox_{\tau G}(x) = \frac{x}{1+\tau \lambda_2}
\end{equation}
and
\begin{equation}\label{eq:prox_F}
    \begin{split}
        \prox_{\tau F^*}(q,p) &= \begin{bmatrix}
            \prox_{\tau \left(\frac{1}{2 \lambda_0}\|\cdot\|_{2}^{2} + \langle \cdot, z \rangle\right)}(q) \\
            \prox_{\tau I_{\norm{\cdot}_{\infty, 2} \leq \lambda_1}}(p)
        \end{bmatrix} \\
        &= \begin{bmatrix}
            \prox_{\frac{\tau}{2 \lambda_0} \|\cdot\|_{2}^{2}}\left(q - \tau z\right) \\
            \prox_{\frac{\tau}{1+\tau \epsilon} I_{\norm{\cdot}_{\infty, 2} \leq \lambda_1}}\left(\tfrac{p}{1+\tau \epsilon}\right)
        \end{bmatrix} \\
        &= \begin{bmatrix}
            \frac{(q-\tau z)}{\left(1+ \tfrac{\tau}{\lambda_0}\right)} \vspace{0.1cm}\\
            \proj_{\norm{\cdot}_{\infty, 2} \leq \lambda_1}\left(\tfrac{p}{1+\tau \epsilon}\right)
        \end{bmatrix}.
    \end{split}
\end{equation}

In the experiment we want to recover the famous Shepp Logan phantom $\hat{x}$ with $400 \times 400$ pixels from measurement with just $40$ equispaced angles and $400$ bins for each angle and a parallel bean geometry, and added 15\% relative Gaussian noise. Hence, the resulting sinogram $z$ is of the shape $40 \times 400$. To implement the mismatch we used non-adjoined implementations of the forward and back-projection. As the forward operator $A$ we use the parallel strip beam projector from the Astra toolbox~\cite{van2016fast} and introduce an adjoint mismatch with the backward projection $V^{*}$ the adjoint of the parallel line beam projector, so that
$$\norm{A-V} \approx 0.2945.$$ All experiments are done in Python 3.12

\paragraph{Existence and Uniqueness of Fixed Points.}
To verify the existence of a fixed point for the primal-dual Douglas-Rachford method with mismatched adjoint, we apply Theorem~\ref{thm:existenceoffixedpoints}. Since both \(G\) and \(F^*\) are proper, strongly convex, and lower semicontinuous functions, and \(A\) and \(V^*\) are bounded linear operators, it remains to verify the inequality~\eqref{eq:conditionexistencefixedpoints}, which in our specific case is given by  
\[
\lambda_2 \min\{\epsilon, \lambda_0^{-1}\} > \frac14 \norm{A-V}^2.
\]
With the choice of \(\lambda_0 = 10\), \(\lambda_2 = 2\), and \(\epsilon = 0.1\), the existence of unique fixed points for the primal-dual Douglas-Rachford method with mismatched adjoint is thus ensured.

\paragraph{Numerical Observations.}
To compute the step sizes for the methods with an mismatched adjoint, we set  
\begin{align*}
\tilde{\mu}_G &= \frac12 \left(\gamma_G + \frac{\norm{A-V}}{2} \sqrt{\frac{\gamma_G}{\gamma_{F^*}}}\right), \;\,\qquad \mu_G = \frac12 \left(\gamma_G+\tilde{\mu}_G\right), \\
\tilde{\mu}_{F^*} &= \frac12 \left(\gamma_{F^*} + \frac{\norm{A-V}}{2} \sqrt{\frac{\gamma_{F^*}}{\gamma_{G}}}\right), \qquad \mu_{F^*} = \frac12 \left(\gamma_{F^*}+\tilde{\mu}_{F^*}\right).
\end{align*}
We choose the relaxation parameter \(\theta = \frac12\) and obtain the step size \(\tau = 0.01965\) using Theorem~\ref{thm:pddrmm-linear-convergence-easy}.  

For solving the inner problem, we use the stochastic gradient descent method with estimated adjoint~\cite{randomdescent}, which we terminate either after 10,000 iterations or when the residual falls below the threshold of \(10^{-6}\). Additionally, we fix \(\lambda_1 = 6\) and keep the remaining parameters of the optimization problem unchanged.  

Comparing the adapted and non-adapted versions of the primal-dual Douglas-Rachford method with mismatched adjoint with its equivalent with a correct adjoint, we obtain the following result: As seen in Figure~\ref{fig:vergleich-pddr}, the adapted primal-dual Douglas-Rachford method with mismatched adjoint essentially behaves like the method with an exact adjoint and converges to a fixed point that is similarly close to the original image. This can also be observed in Figure~\ref{fig:bilder-pddr}. The adapted and non-adapted versions of the primal-dual Douglas-Rachford method with mismatched adjoint share the same fixed point but converge at different speeds. Compared to the Chambolle-Pock method with mismatched adjoint~\cite{cp_mismatched}, no clear advantage is observed for the primal-dual Douglas-Rachford method with mismatched adjoint. However, we like to mention that the condition $\gamma_G \gamma_F^* > 2 \norm{A-V}^2$ for the Chambolle-Pock method with mismatched adjoint to provably function is eight times higher than the bound provided for the convergence of the primal-dual Douglas-Rachford method with an adjoint mismatch. Additionally, noting that the duration of an iteration strongly depends on the speed of solving the linear system, the Chambolle-Pock method with mismatched adjoint may be advantageous when no further information about the structure of the linear mappings is available.

\begin{figure}
    \centering
    \includegraphics[width=0.45\textwidth]{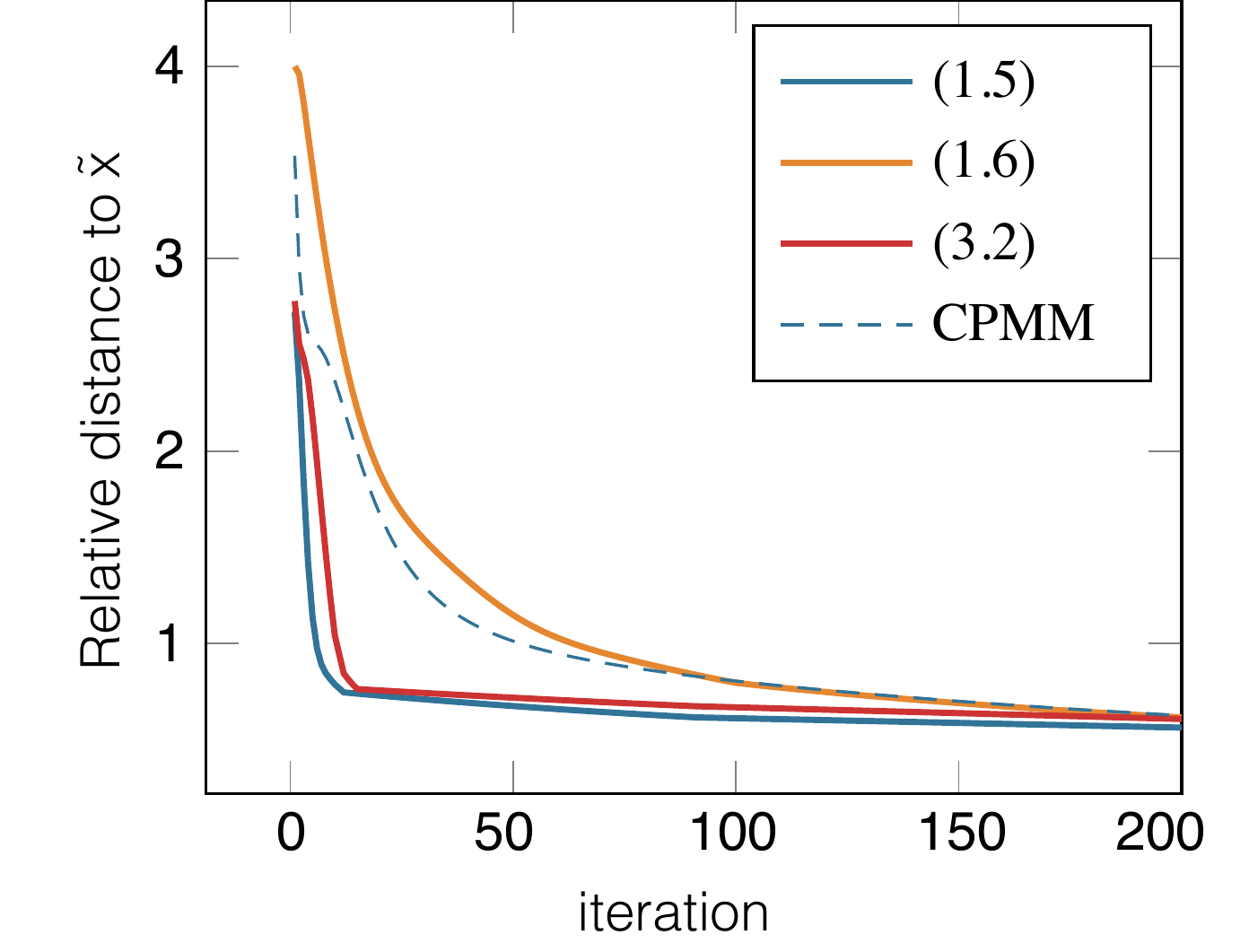}\quad
    \includegraphics[width=0.45\textwidth]{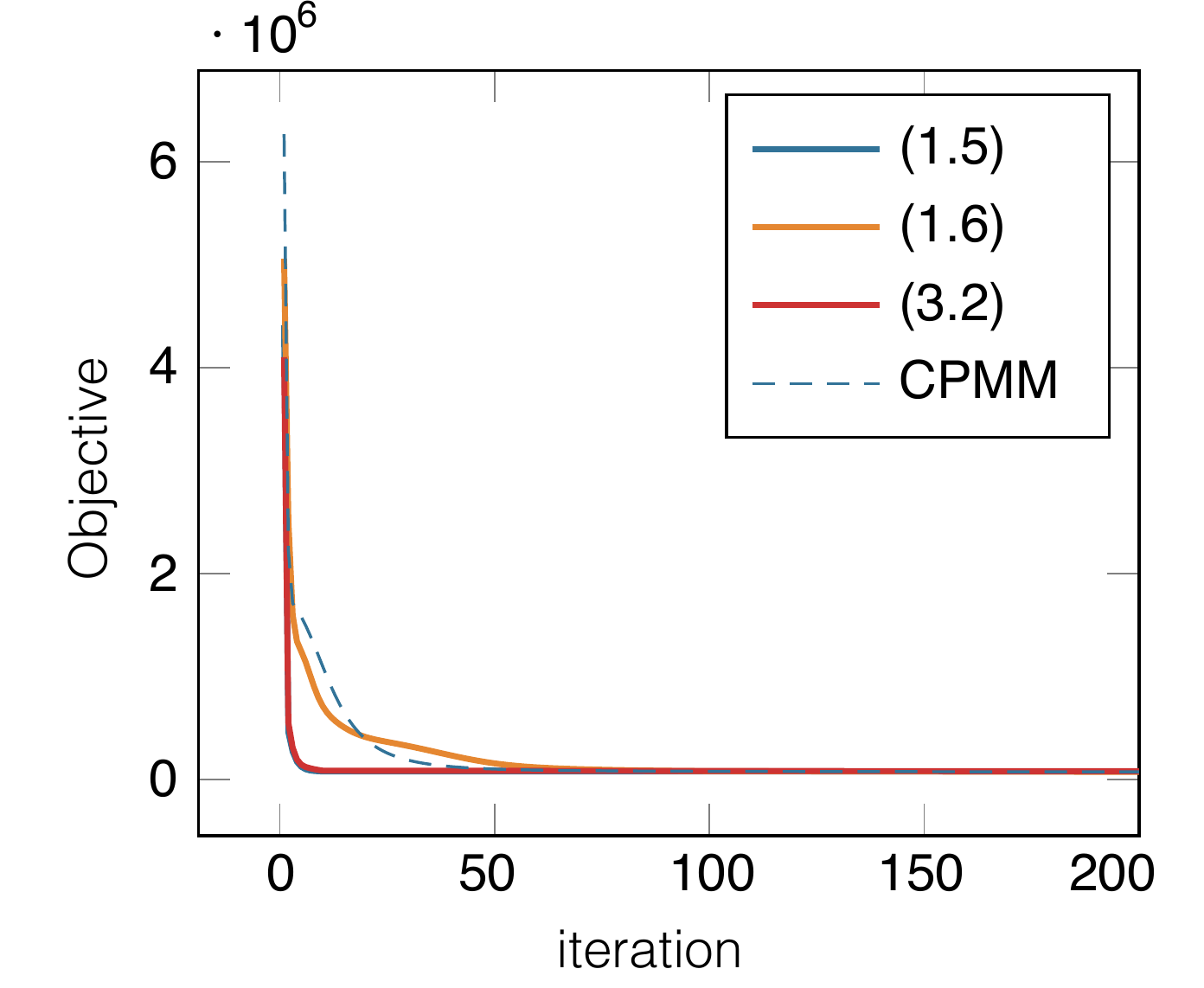}
    \caption{Comparison of the primal-dual Douglas-Rachford method (blue line) with the (adapted) primal-dual Douglas-Rachford method with mismatched adjoint. The adapted version is shown in red, and the non-adapted version in orange. As another reference, the Chambolle-Pock method with mismatched adjoint is represented by the dashed blue line. \newline\noindent Left: Relative distance to \(\tilde{x}\) over the iterations. \newline\noindent  Right: Decrease in the value of the primal objective function over the iterations.}
    \label{fig:vergleich-pddr}
\end{figure}

\begin{figure}[htb]
\centering\scriptsize
\begin{tabular}{rcccc}
&  Original & Sinogram & \shortstack{Reconstruction \\ with FBP} & \\
& \includegraphics[width=1.75cm]{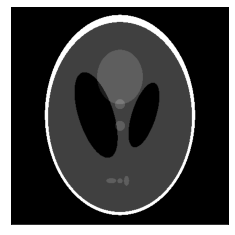} &
 \includegraphics[height=1.75cm, width=1.75cm]{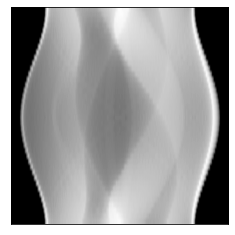} & \includegraphics[width=1.75cm]{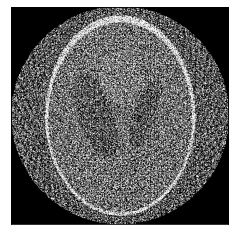} & 
 \\
 & \shortstack{Reconstruction \\ with \eqref{alg-part:pddr}}   &  \shortstack{Reconstruction Error \\ with \eqref{alg-part:pddr}} & \shortstack{Reconstruction  \\ with \cite{cp_mismatched}} & \shortstack{Reconstruction Error  \\ with \cite{cp_mismatched}} \\
& \includegraphics[height=1.75cm]{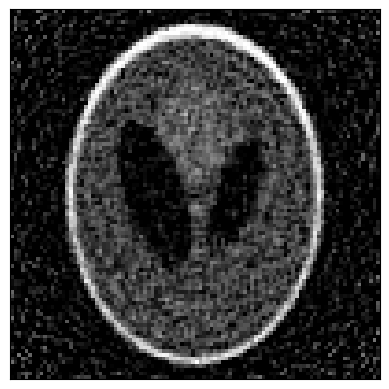}& 
\includegraphics[height=1.75cm]{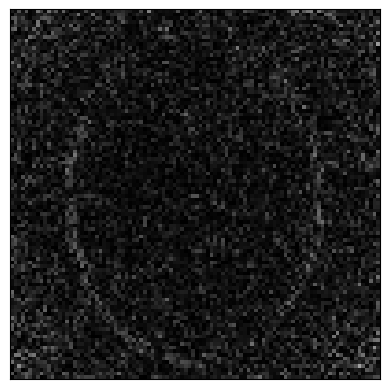} & 
\includegraphics[height=1.75cm]{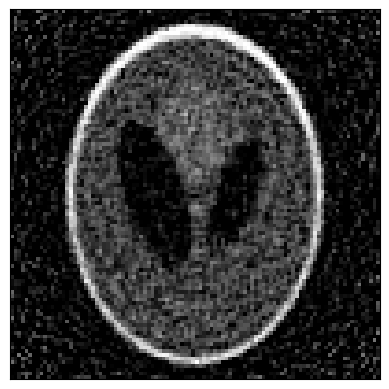} &
\includegraphics[height=1.75cm]{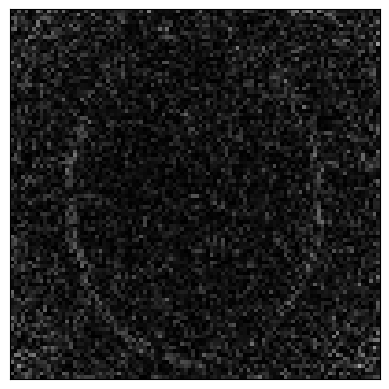} \\
& \shortstack{Reconstruction \\ with \eqref{alg:adapted-pddr}}   &  \shortstack{Reconstruction Error \\ with \eqref{alg:adapted-pddr}} & \shortstack{Reconstruction  \\ with \eqref{alg-part:pddr_mismatch}} & \shortstack{Reconstruction Error  \\ with \eqref{alg-part:pddr_mismatch}} \\
&
\includegraphics[height=1.75cm]{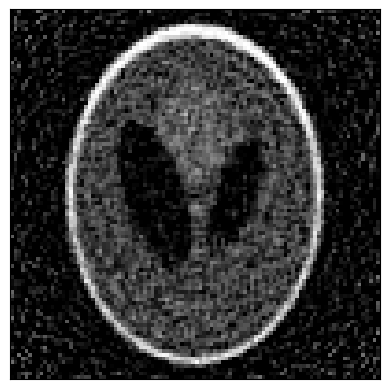}& 
\includegraphics[height=1.75cm]{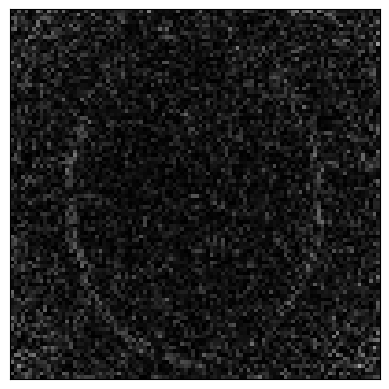} & 
\includegraphics[height=1.75cm]{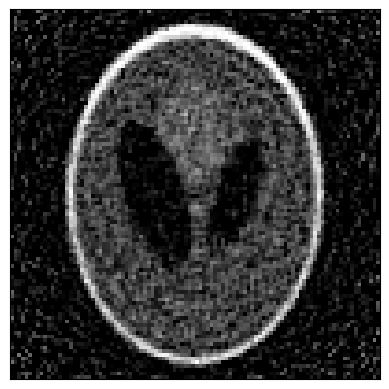} &
\includegraphics[height=1.75cm]{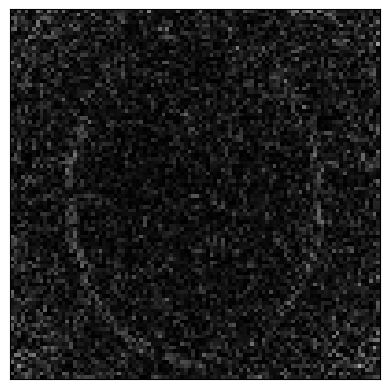} \\
 
\end{tabular}
\caption{Reconstruction of the Shepp-Logan phantom. Top Right: Reconstruction with the Filtered Backprojection. Middle left: Reconstruction using the primal-dual Douglas-Rachford method with an exact adjoint. Middle right: Reconstruction using the Chambolle-Pock method with mismatched adjoint. Bottom left: Reconstruction using the adapted primal-dual Douglas-Rachford method with mismatched adjoint. Bottom right: Reconstruction using the primal-dual Douglas-Rachford method with mismatched adjoint. All images use a fixed grayscale range with values from \(0.0\) to \(1.0\).}
\label{fig:bilder-pddr}
\end{figure}

\paragraph{Comparison to Alternative Methods with Mismatched Adjoint.}
Another aspect of interest is the comparison between the (adapted) primal-dual Douglas-Rachford method with alternative approaches. Therefore, we compare these three algorithms with the Condat-Vũ method and the Loris-Verhoeven method with mismatched adjoint from~\cite{Chouzenoux2023ConvergenceRF}. While only weak convergence results for specific objective functions have been proven for these methods so far, the objective function used in our experiment falls within the scope of their analysis.  

For the Condat-Vũ method with mismatched adjoint we take \(\tau = 0.0001\), \(\sigma = 0.25\), and \(\theta = 1\) to obtain valid step sizes and relaxation parameters according to~\cite[Cor. 3.8]{Chouzenoux2023ConvergenceRF}.  

Similarly, for the iteration of the Loris-Verhoeven method with mismatched adjoint we again take \(\tau = 0.0001\), \(\sigma = 0.25\), and \(\theta = 1\) to obtain valid step sizes and relaxation parameters according to~\cite[Prop. 4.2]{Chouzenoux2023ConvergenceRF}. Since both methods converge more slowly in this experiment than the previously analyzed methods, we performed 10,000 iterations of each algorithm. The resulting reconstructions are shown in Figure~\ref{fig:bilder-vergleich}. We observe that all methods produce similar results. However, it is worth noting that the error in the result of the Condat-Vũ method with mismatched adjoint is larger at the object's boundary compared to the reconstructions obtained from the other methods.  

\begin{figure}[htb]
\centering\scriptsize
\begin{tabular}{rcccc}
&  Original & Sinogram & \shortstack{Reconstruction \\ with FBP} & \\
& \includegraphics[width=1.75cm]{figures/fig2_01.png} &
 \includegraphics[height=1.75cm, width=1.75cm]{figures/fig2_02.png} & \includegraphics[width=1.75cm]{figures/fig2_03.png} &  
 \\
 & \shortstack{Reconstruction \\ with \eqref{alg-part:pddr_mismatch}} & \shortstack{Reconstruction Error  \\ with \eqref{alg-part:pddr_mismatch}} & \shortstack{Reconstruction \\ with \eqref{alg:adapted-pddr}}   &  \shortstack{Reconstruction Error \\ with \eqref{alg:adapted-pddr}}  \\
 &
\includegraphics[height=1.75cm]{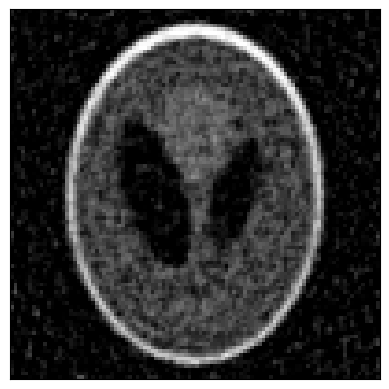}& 
\includegraphics[height=1.75cm]{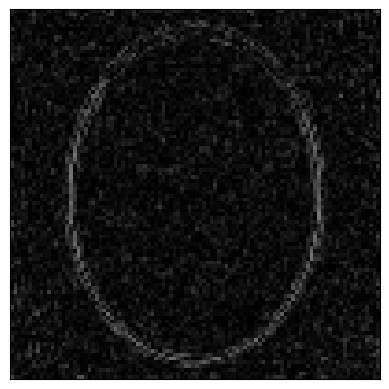} & 
\includegraphics[height=1.75cm]{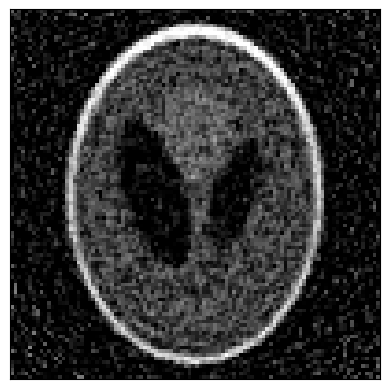} &
\includegraphics[height=1.75cm]{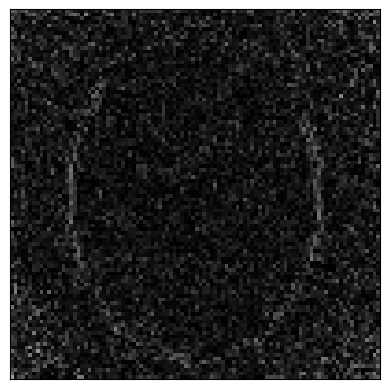} \\
 & \shortstack{Reconstruction \\ with CV}   &  \shortstack{Reconstruction Error \\ with CV} & \shortstack{Reconstruction  \\ with LV} & \shortstack{Reconstruction Error  \\ with LV}  \\
& 
\includegraphics[height=1.75cm]{figures/image_pddrm.png} &
\includegraphics[height=1.75cm]{figures/image_pddrm_diff.png} & 
\includegraphics[height=1.75cm]{figures/image_apddr.png}& 
\includegraphics[height=1.75cm]{figures/image_apddr_diff.png} 
\\
\end{tabular}
\caption{Reconstruction of the Shepp-Logan phantom with reconstruction error. Top Right: Reconstruction with the Filtered Backprojection. Middle left: Reconstruction using the primal-dual Douglas-Rachford method with mismatched adjoint. Middle right: Reconstruction using the adapted primal-dual Douglas-Rachford method with mismatched adjoint. Bottom left: Reconstruction using the Condat-Vũ method with mismatched adjoint (CV). Bottom center: Reconstruction using the Loris-Verhoeven method with mismatched adjoint (LV). All images use a fixed grayscale scale with values from \(0.0\) to \(1.0\).}
\label{fig:bilder-vergleich}
\end{figure}

In Figure~\ref{fig:vergleich-all}, we visualize the distance to the Shepp-Logan phantom as well as the value of the objective function to be minimized over the first 1,000 iterations, where applicable. We observe that the adapted primal-dual Douglas-Rachford method with mismatched adjoint is close to the original image after only a few iterations. Meanwhile, the Chambolle-Pock method with mismatched adjoint and the primal-dual Douglas-Rachford method with mismatched adjoint produce similar results. Notably, the methods analyzed in this paper converge significantly faster in this example than the Condat-Vũ method and the Loris-Verhoeven method with mismatched adjoint, for which no step size strategy is provided in~\cite{Chouzenoux2023ConvergenceRF}.  

\begin{figure}
    \centering
    \includegraphics[width=0.9\textwidth]{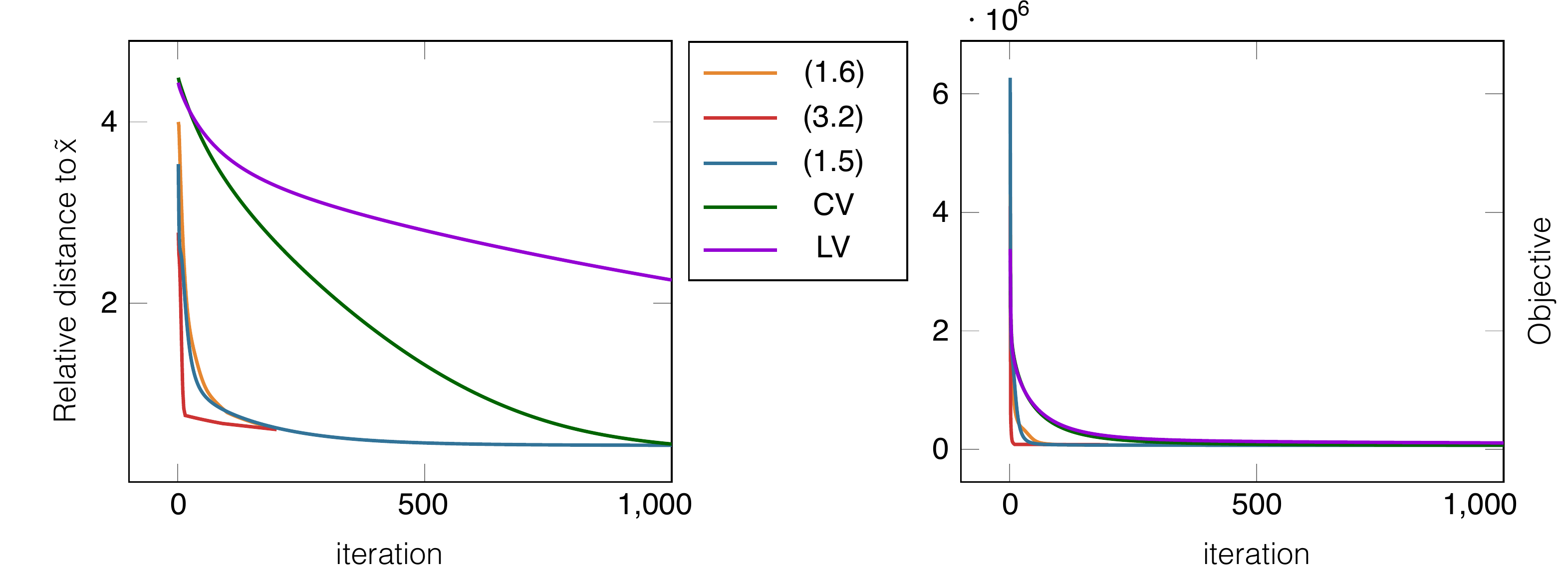}
    \caption{Comparison of different primal-dual methods with mismatched adjoint. The results of the primal-dual Douglas-Rachford method with mismatched adjoint (orange line), the adapted primal-dual Douglas-Rachford method with mismatched adjoint (red line), the Chambolle-Pock method with mismatched adjoint (blue line), the Condat-Vũ method with mismatched adjoint (green line), and the Loris-Verhoeven method with mismatched adjoint (violet line) are visualized. \newline\noindent 
    Left: Relative distance to \(\tilde{x}\) over the iterations. \newline\noindent 
    Right: Decrease in the value of the primal objective function over the iterations.}
    \label{fig:vergleich-all}
\end{figure}

\section{Conclusion}
\label{sec:conclusion}
In this paper we studied the influence of an adjoint mismatch on the primal-dual Douglas-Rachford method. Under mild conditions, we established a method that preserves the decomposition into two monotone operators and hence is covered by previous results. 
Additionally we studied the convergence of the primal-dual Douglas-Rachford method under the influence of an adjoint mismatch. For both methods we provided simple conditions that guarantee the existence of a unique fixed point and an upper bound on the distance between the original solution and the fixed point of iteration with mismatch.
Furthermore, we established step sizes that guarantee linear convergence of the primal-dual Douglas-Rachford method with mismatched adjoint. Thus, approximating the adjoint with a computationally more efficient algorithm can be done as long as the assumptions are respected. Furthermore, we discussed advantages and disadvantages of our analysis and illustrated our results on an example from computerized tomography, in which the mismatched adjoint was obtained using different discretization schemes for the forward operator and the mismatched adjoint.

\section*{Acknowledgments}
We sincerely want to thank Prof. Dr. D. Lorenz for his enriching comments. They have made an essential contribution to improving the quality of the paper. Furthermore we like to thank Dr. S. Banert for the help on establishing an appropriate convergence rate. E.N. is supported by the MUR Excellence Department Project awarded to the department of mathematics at the University of Genoa, CUP D33C23001110001, and by the US Air Force Office of Scientific Research (FA8655-22-1-7034). E.N. is a member of GNAMPA of the Istituto Nazionale di Alta Matematica (INdAM).

\printbibliography
\end{document}